\documentclass[amstex,a4]{article}

 \usepackage{amsfonts}
 \usepackage{amsmath}
 \usepackage{amsthm}
 \usepackage{amssymb}
 \usepackage{graphicx}

 \allowdisplaybreaks[4]

 \newtheorem{theorem}{Theorem}[section]
 \newtheorem{corollary}{Corollary}[section]
 
 \newtheorem{lemma}{Lemma}[section]

\begin{document}
 \title{\textbf{Initial-boundary value problems for linear diffusion equation 
with multiple time-fractional derivatives}}

\author{\textbf{Zhiyuan Li$^1$}, \textbf{Masahiro Yamamoto$^2$}}

 \date{}

 \maketitle

 \begin{center}
 $^{1,2}$Graduate School of Mathematical Sciences,

 The University of Tokyo

 E-mail: $^{1}$zyli@ms.u-tokyo.ac.jp,\ $^{2}$myama@ms.u-tokyo.ac.jp
 \end{center}
 
\begin{abstract}
In this paper, we discuss initial-boundary value problems for linear diffusion 
equation with multiple time-fractional derivatives. By means of the 
Mittag-Leffler function and the eigenfunction
expansion, we reduce the problem to an integral equation for a solution, 
and we apply the fixed-point
theorem to prove the unique existence and the H\"older regularity of solution. 
For the case of the
homogeneous equation, the solution can be analytically extended to a sector 
$\{z\in\mathbb{C};z\neq0,|\arg z|<\frac{\pi}{2}\}$. 
In the case where all the coefficients of the 
time-fractional derivatives are positive constants, by the Laplace transform and 
the analyticity, we can prove that if a function satisfies the fractional diffusion equation
and the homogeneous Neumann boundary condition on arbitrary subboundary 
as well as the homogeneous Dirichlet boundary condition on the whole 
boundary, then it vanishes identically.
 \end{abstract}


\section{Introduction}
We assume $\Omega$ to be a bounded domain in $\mathbb{R}^d$ with sufficiently 
smooth boundary $\partial \Omega$.
 We consider an initial-boundary value problem for a diffusion equation with 
multiple fractional time derivatives:
 \begin{equation} \label{multifrac}
   \left\{ {\begin{array}{*{20}c}
  { \partial_t^{\alpha_1} u(t) 
+ \sum_{j=2}^{\ell} q_j \partial_t^{\alpha_j} u(t)
   =-\mathcal{A}u(t) + B(x)\cdot \nabla u(t) + F(x,t) },\ t>0, 
\hfill \\
  {u(x,0)=a,\ x\in \Omega,} \hfill \\
  {u(x,t)=0,\ x\in \partial\Omega,\ t \in (0,T).} \hfill
 \end{array} } \right.
 \end{equation}
Here $0<\alpha_{\ell} < \cdots < \alpha_2 < \alpha_1 <1$.
For $\alpha\in(0,1)$, by $\partial_t^{\alpha}$ 
we denote the Caputo fractional derivative with respect to $t$:
 \[
\partial_t^{\alpha} g(t) = \frac{1}{\Gamma(1-\alpha)} 
\int_0^t (t-\tau)^{-\alpha}\frac{d}{d\tau} g(\tau) d\tau 
 \]
and $\Gamma$ is the Gamma function. See, e.g., Podlubny \cite{Podlubny} 
and Kilbas et al \cite{Kil} for the definition and properties of the Caputo
derivative.

The operator $\mathcal{A}$ denotes a second-order partial differential 
operator in the following form
 \[
 (-\mathcal{A}u)(x)
= \sum_{i,j=1}^d \frac{\partial}{\partial x_i}\left(a_{ij}(x)
\frac{\partial}{\partial x_j}u(x)\right)
          + b(x)u(x),\ x\in \Omega,
 \]
for $u\in H^2(\Omega) \cap H_0^1(\Omega)$, 
and we assume that $a_{ij}=a_{ji}\in C^1(\overline{\Omega})$, $1\le i,j\le d$, 
$b\in C(\overline{\Omega})$, $b(x) \leq 0$ 
 for $x\in \overline{\Omega}$ and that there exists a constant $\nu>0$ 
such that
 \[
\nu \sum_{j=1}^d \xi_j^2 \le \sum_{j,k=1}^d a_{jk}(x) \xi_j\xi_k, 
\ x\in\overline{\Omega},\ \xi\in\mathbb{R}^d.
 \]

The classical diffusion equation with integer-order derivative has played 
important roles in modelling contaminants diffusion processes. 
However, in recent two decades, more experimental data in some  
diffusion processes in highly heterogeneous media, show that 
the classical model may be inadequate in order to interpret experimental data.
For example, Adams and Gelhar \cite{AdGe} points out that field data in 
a saturated zone of a highly heterogeneous aquifer indicate a long-tailed 
profile in the spatial distribution of densities as the time passes, which 
is difficult to be interpreted by the classical diffusion equation.
The above phenomenon of long-tailed profile has been investigated by 
many researchers, and see Berkowitz et al \cite{BSS}, Giona et al \cite{GGR}, 
Y. Hatano and N. Hatano \cite{Hatano}, and the references therein. 
For better model equation, an equation where the first-order time derivative 
is replaced by a derivative of fractional order $\alpha\in 
(0,1)$, has been proposed.
As defined below, the fractional derivative possesses the memory effect 
and leads to realization of slow diffusion. 
This modified model is presented as a useful approach for the description 
of transport dynamics in complex system that are governed by anomalous 
diffusion and non-exponential relaxation patterns,
and attracted great attention from different areas. For numerical calculation, 
see Beson et al \cite{BSMW}, Meerschaert et al \cite{MT}, 
Diethelm and Luchko \cite{DiLu} and the references therein. 
For theoretical aspects, see Gorenflo et al \cite{GLZ}, Hanyaga \cite{Hanyaga},
Luchko and Gorenflo \cite{Luchko00}, 
Luchko \cite{Luchko0,Luchko,Luchko1}, Sakamoto and Yamamoto 
\cite{SakYam}, Xu et al \cite{XCY}, etc. 
From a viewpont of the stochastic analysis, one can regard the time-fractional 
diffusion equation as a macroscopic model derived from the continuous-time 
random walk. Metzler and Klafter \cite{MK}
demonstrated that a fractional diffusion equation describes a non-Markovian 
diffusion process with a memory. Roman and Alemany \cite{RA} investigated 
continuous-time random walks on fractals and showed that the average 
probability density of random walks on fractals obeys a diffusion equation 
with a fractional time derivative asymptotically.

In some recent publications such as e.g. \cite{Chechkin1, Chechkin2, Sokolov},
the time-fractional diffusion equations of distributed order is investigated. 
A distributed order derivative is an integral of fractional derivatives 
with respect to continuously changing orders.
Differential equations of the distributed order and some of their applications 
were considered e.g., in \cite{Kochubei, Meerschaert2, Umarov}. 
One important particular case of the time-fractional diffusion
equation of distributed order is that the weight function is taken 
in form of a finite linear combination of the Dirac $\delta$-functions with 
the positive weight coefficients (see e.g., \cite{Luchko0, Luchko2}). 
This yields the so-called diffusion equation with multiple time-fractional 
derivatives, which is studied in our paper. 
As for diffusion equations with multiple fractional time derivatives, 
see also Jiang et al \cite{JLTB}, Gejji et al \cite{DGB}, 
and the the references therein. 
The article \cite{JLTB} discusses the case where the spatial dimension is one,
the coefficients are constant and the spatial fractional derivative is 
considered, and establishes the formula of the solution. 
In the paper \cite{DGB}, a solution to an initial-boundary value problem 
is formally represented by Fourier series and the multivariate Mittag-Leffler 
function.  As for multivariate Mittag-Leffler functions, see e.g., 
\cite{Luchko00}.
However no proofs for the convergence of the series and for the uniqueness 
of the solution are given in \cite{DGB}. A proof of the convergence of the 
series defining the solution of the more general distributed
order fractional Cauchy problems in bounded domains can be found in the paper 
\cite{Meer}.  The paper \cite{Luchko2} proves the unique existence of the
solution, the maximum principle and related properties in the case where 
the coefficients of the time derivatives are positive and depedenent on 
$x$, and the arguments are based on the Fourier method, that is,
the separation of the variables.

These papers mainly discuss the case where the spatial differential 
operators is a symmetric elliptic elliptic operator.

The paper \cite{SB} proves the uniqueness and the regularity of solution 
to an initial-boundary problem for a symmetric fractional diffusion 
equation with two terms of time-fractional derivatives, by assuming the 
existence of the solution.  The method is similar to \cite{Luchko00}
and \cite{SakYam}.

In this article, following \cite{Luchko2} and \cite{SB}, we deal with 
the initial-boundary value problems for linear diffusion equation 
with multiple time-derivatives. The difference is that here 
we investigate the linear non-symmetric diffusion equation with 
the variable coefficients of fractional 
time derivatives not necessary constant or positive variable. 
Such kind of equation simulates the advection and so can be regarded as
more feasible model equation than symmetric fractional diffusion equations 
in modelling diffusion in porous media.

The rest of this article is organized as follows:\\ 
Section 2: The fixed point method is applied to prove 
unique existence as well as regularity of solution to (\ref{multifrac}).\\
Section 3: Based on the existence result in section 2, we prove 
the regularity of H\"{o}lder of the solution step by step from the 
continuous regularity to the H\"{o}lder regularity with index $\theta$ 
under the assumption that initial condition $a=0$ and the source term 
$F\in C^{\theta}([0,T];L^2(\Omega))$ with the compatibility condition 
$F(0)=0$.\\
Section 4: We prove that the solution can be analytically 
extended to a sector $\{z\in\mathbb{C};z\neq0,|\arg z|<\frac{\pi}{2}\}$ 
when $F=0$ and $a\in L^2(\Omega)$.
\\
Section 5: The analyticity of the solution and the Laplace 
transform are applied to show that 
if a function satisfies the fractional diffusion equation and 
the homogeneous Neumann boundary condition on arbitrary subboundary 
as well as the homogeneous Dirichlet boundary condition on the whole 
boundary, then it vanishes identically.  This is a weak type of 
unique continuation.


 \section{Existence, uniqueness and regularity of the solution}
 Let $L^2(\Omega)$ be a usual $L^2$-space with the inner product 
$(\cdot,\cdot)$, and $H^{k}(\Omega)$, 
 $H_0^1(\Omega)$ denote Sobolev spaces (e.g., Adams \cite{Adams}). 
 We set $\|a\|_{L^2(\Omega)}=(a,a)^{\frac{1}{2}}$.
 
 We define an operator $A$ in $L^2(\Omega)$ by
 \[
(Au)(x)=(\mathcal{A}u)(x),\ x\in\Omega,\ D(A)=H^2(\Omega) \cap H_0^1(\Omega).
 \]
 Then the fractional power $A^{\gamma}$ is defined for $\gamma\in\mathbb{R}$(e.g., \cite{Pazy}), 
 and $D(A^{\gamma})\subset H^{2\gamma}(\Omega)$, $D(A^{\frac{1}{2}})=H_0^1(\Omega)$ for example. 
 We note that $\|u\|_{D(A^{\gamma})}:=\|A^{\gamma}u\|_{L^2(\Omega)}$ is stronger than 
 $\|u\|_{L^2(\Omega)}$ for $\gamma>0$. 
 
Since $A$ is a symmetric uniformly elliptic operator, the spectrum of $A$ is entirely composed of eigenvalues
and counting according to the multiplicities, we can set $0<\lambda_1 \le \lambda_2 \le \cdots$.
By $\phi_n \in D(A)$, we denote the orthonormal eigenfunction corresponding to $\lambda_n$: 
 $A\phi_n=\lambda_n\phi_n$. Then the sequence $\{\phi_n\}_{n\in\mathbb{N}}$ is orthonormal basis in $L^2(\Omega)$.
 Then we see that 
 \[
D(A^{\gamma}) 
= \left\{\psi\in L^2(\Omega): \sum_{n=1}^{\infty} \lambda_n^{2\gamma} |(\psi,\phi_n)|^2<\infty\right\} 
 \]
 and that $D(A^{\gamma})$ is a Hilbert space with the norm
 \[
\|\psi\|_{D(A^{\gamma})}=\left( \sum_{n=1}^{\infty} \lambda_n^{2\gamma} |(\psi,\phi_n)|^2 \right)^{\frac{1}{2}}.
 \]
 Moreover we define the Mittag-Leffler function by
 \[
E_{\alpha,\beta}(z):=\sum_{k=0}^{\infty}\frac{z^k}{\Gamma(\alpha k+\beta)},\ z\in\mathbb{C},
 \]
 where $\alpha>0$ and $\beta\in\mathbb{R}$ are arbitrary constants. By the power series, we can directly verify that
 $E_{\alpha,\beta}(z)$ is an entire function of $z\in\mathbb{C}$.
 
 Now we define an operator $S(t):L^2(\Omega)\rightarrow L^2(\Omega)$ for $t>0$ by
\begin{align}\label{S(t)}
S(t)a:=\sum_{n=1}^{\infty} (a,\phi_n) E_{\alpha_1,1}(-\lambda_nt^{\alpha_1}) \phi_n \ \text{in}\ L^2(\Omega) 
\end{align}
 for $a\in L^2(\Omega)$. Moreover the term-wise differentiations are possible and give
\begin{align*}
S'(t)a
&:=-\sum_{n=1}^{\infty}\lambda_n(a,\phi_n)t^{\alpha_1-1}E_{\alpha_1,\alpha_1}(-\lambda_nt^{\alpha_1})\phi_n\ 
\text{in}\ L^2(\Omega)\\
S''(t)a
&:=-\sum_{n=1}^{\infty}\lambda_n(a,\phi_n)t^{\alpha_1-2}E_{\alpha_1,\alpha_1-1}(-\lambda_nt^{\alpha_1})\phi_n\ 
\text{in}\ L^2(\Omega)
\end{align*}
for $a\in L^2(\Omega)$, $t>0$. Moreover, it is known (See, e.g., Theorem 1.6 in \cite{Podlubny}) that there 
exists constant $C>0$ such that 
\begin{align} 
\|A^{\gamma-1}S'(t)\|&\leq C t^{\alpha_1-1-\alpha_1 \gamma},\ 0<t\leq T,\ 0\leq\gamma\leq 1. \label{Ar-1s'}\\
\|A^{\gamma-1}S''(t)\|&\leq C t^{\alpha_1-2-\alpha_1 \gamma},\ 0<t\leq T,\ 0\leq\gamma\leq 1 \label{Ar-1s''},
\end{align}
where $\|\cdot\|$ denotes the operator norm from $L^2(\Omega)$ to $L^2(\Omega)$.\\
 
Now we are ready to state our first main result.

 \begin{theorem}[a priori estimate] \label{Estimate}
Suppose that $a\in L^2(\Omega)$, $F\in L^{\infty}(0,T; L^2(\Omega))$, 
$B(x):=(B_1(x),\cdots,B_d(x))$, 
$B_i \in W^{2,\infty}(\Omega)$, $1\leq i \leq d$, $q \in W^{2,\infty}(\Omega)$.
Let $0<\alpha_{\ell}<\cdots < \alpha_1 <1$ and $u(t)\in D(A^{\gamma})$, $0<t\leq T$ 
satisfy (\ref{multifrac}).
Then
 \[
   \|u(t)\|_{H^{2\gamma}(\Omega)} 
  \leq C \left(t^{-\alpha_1 \gamma} \|a\|_{L^2(\Omega)}
+\|F\|_{L^{\infty}(0,T; L^2(\Omega))}\right), 
  0 < t \leq T,
 \]
 where $\gamma\in[\frac{1}{2},1)$ and $C>0$ is a constant which is independent 
of $a$, $F$ in (\ref{multifrac}), but may depend on $T$, $d$, $\{\alpha_{j}\}_{j=1}^{\ell}$, 
$\gamma$, $\Omega$ and the
coefficients of the operator $\mathcal{A}$, $\{q_i\}_{i=2}^{\ell}$.
\end{theorem}
In \cite{SB} a similar fractional diffusion equation is discussed for 
$F=0$ and $B=0$ and a similar regularity is proved.  However \cite{SB}
assumes an extra condition $\alpha_1+\alpha_{\ell} > 1$, and our main 
result needs not such an assumption. 
\begin{proof}
Since $u$ is the solution of our initial-boundary value problem (\ref{doublefrac}),
by an argument similar to the proof of Theorem 1 in \cite{SB}, we find
\begin{align} \label{solution aF}
u(t)
=&-\int_0^t A^{-1}S'(t-\tau)(B\cdot\nabla u+F)d\tau 
\ +\sum_{i=2}^{\ell}\frac{1}{\Gamma(1-\alpha_i)}\int_0^t A^{-1}S'(t-\tau)(t-\tau)^{-\alpha_i}(q_iu(\tau))d\tau\nonumber\\
&+\sum_{i=2}^{\ell}\frac{1}{\Gamma(1-\alpha_i)}\int_0^t \int_0^{t-\tau} A^{-1}S''(t-\eta-\tau)(\eta^{-\alpha_i}-(t-\tau)^{-\alpha_i})d\eta q_iu(\tau) d\tau\nonumber\\
&-\sum_{i=2}^{\ell}\frac{1}{\Gamma(1-\alpha_i)}\int_0^t A^{-1}S'(t-\tau)\tau^{-\alpha_i}q_iad\tau+S(t)a
=: \sum_{k=1}^5 I_k(t),\ 0<t\leq T.
\end{align}

Now let us turn to the evaluation of the integral equation (\ref{solution aF}). For this purpose, 
taking the operator $A^{\gamma}(\gamma \in [\frac{1}{2},1))$ on the both sides of (\ref{solution aF}), 
we estimate each of the five terms separately.

By the fact $D(A^{\alpha})\subset D(A^{\beta})(\forall \alpha\geq\beta\geq0)$, noting that $D(A^{\frac{1}{2}})=H_0^1(\Omega)$, it follows that 
 \[
\|B\cdot\nabla u\|_{L^2(\Omega)}\leq C\|u\|_{H_0^1(\Omega)}\leq C_1\|u\|_{D(A^{\frac{1}{2}})}
\leq C_2\|u\|_{D(A^{\gamma})}, \forall \gamma\in[\frac{1}{2},1).
 \]
Moreover, since $q_j\in W^{2,\infty}(\Omega)$, $j=2,\cdots,\ell$ by the interpolation theory
(see, e.g., \cite{Lions}), 
we have $\|A^{\gamma}(q_ju)\|_{L^2(\Omega)}\leq C\|A^{\gamma}u\|_{L^2(\Omega)}$
for any $\gamma\in[0,1]$ and any $u\in D(A^{\gamma})$. 
Therefore for $I(t):=I_1(t) + I_2(t) + I_4(t)+I_5(t)$, $\forall 0<t\leq T$, using (\ref{Ar-1s'}) and
(\ref{Ar-1s''}), we have the following estimate
 \begin{align}
   \|I\|_{L^2(\Omega)}
\leq& C\|a\|_{L^2(\Omega)} (t^{-\alpha_1 \gamma} 
      + \sum_{i=2}^{\ell}t^{\alpha_1-\alpha_1\gamma-\alpha_i})
     +C \int_0^t (t-\tau)^{\alpha_1 - \alpha_1\gamma-1} \|F(\tau)\|_{L^2(\Omega)} d\tau\nonumber\\
    &+C \int_0^t (t-\tau)^{\alpha_1 - \alpha_1\gamma-1} \|A^{\gamma}u(\tau)\|_{L^2(\Omega)} d\tau
     +C \sum_{i=2}^{\ell} \int_0^t (t-\tau)^{\alpha_1 - 1 - \alpha_i} \|A^{\gamma}u(\tau)\|_{L^2(\Omega)}d\tau.\nonumber
 \end{align}
In order to evaluate $I_3(t)$ we need more technical treatment to the integral 
in $I_3(t)$. In fact, after the change of variable $\tilde{\tau}=t-\tau$, and letting $\tilde{\eta}=\dfrac{\eta}{\tilde{\tau}}$, we have
\begin{align}
\|I_3(t)\|_{L^2(\Omega)}
\leq& C\sum_{i=2}^{\ell} \int_0^t\Big[\int_0^1 (\tilde{\tau}-\tilde{\tau}\tilde{\eta})^{\alpha_1-2}((\tilde{\tau}\tilde{\eta})^{-\alpha_i}-\tilde{\tau}^{-\alpha_i})\tilde{\tau}d\tilde{\eta}
\Big]\|A^{\gamma}u(t-\tilde{\tau})\|_{L^2(\Omega)}d\tilde{\tau}\nonumber\\
=& C\sum_{i=2}^{\ell}\int_0^t\Big[\int_0^1 (1-\eta)^{\alpha_1-2}(\eta^{-\alpha_i}-1) d\eta
+1\Big]\tau^{\alpha_1-\alpha_i-1}\|A^{\gamma}u(t-\tau)\|_{L^2(\Omega)}d\tau\nonumber
\end{align}
Finally, we are to prove 
$J_i:=\int_0^1 (1-\eta)^{\alpha_1-2}(\eta^{-\alpha_i}-1) d\eta<\infty$, $j=2,\cdots,\ell$.
In fact, we represent $J_i$ as follows:
\begin{align}
J_i=\int_0^{\frac{1}{2}} (1-\eta)^{\alpha_1-2}(\eta^{-\alpha_i}-1) d\eta
  +\int_{\frac{1}{2}}^1 (1-\eta)^{\alpha_1-2}(\eta^{-\alpha_i}-1) d\eta=:J_i'+J_i''.\nonumber
\end{align}
For $J_i'$. Using the inequality $(1-\eta)^{\alpha_1-2}\leq (\frac{1}{2})^{\alpha_1-2}$, 
$\forall \eta\in[0,\frac{1}{2}]$, we derive
\[
J_i'
\leq C\int_0^{\frac{1}{2}} (\eta^{-\alpha_i}-1) d\eta <  \infty.
\]
For $J_i''$. Using the inequality $\eta^{-\alpha_i}-1\leq C(1-\eta)\eta^{-\alpha_i-1}$, 
$\forall \eta\in(0,1)$, $j=2,\cdots,\ell$. 
This inequality is proved, for example, by means of the mean value theorem. Then we can deduce
\[
J_i''
\leq C\int_{\frac{1}{2}}^1 (1-\eta)^{\alpha_1-2}(1-\eta)\eta^{-\alpha_i-1} d\eta
=C\int_{\frac{1}{2}}^1 (1-\eta)^{\alpha_1-1}\eta^{-\alpha_i-1} d\eta
<  \infty.
\]
Collecting the estimates for $J_i'$ and $J_i''$, we have
\begin{align}
\|I_3(t)\|_{L^2(\Omega)} 
\leq C\sum_{i=2}^{\ell} \int_0^t\tau^{\alpha_1-\alpha_i-1}\|A^{\gamma}u(t-\tau)\|_{L^2(\Omega)}d\tau
=C\sum_{i=2}^{\ell}\int_0^t(t-\tau)^{\alpha_1-\alpha_i-1}\|A^{\gamma}u(\tau)\|_{L^2(\Omega)}d\tau.\nonumber
\end{align}
Finally $A^{\gamma}u$ can be estimated as follows: 
for $0<t\leq T$, $\gamma\in [\frac{1}{2},1)$,
 \begin{align}
    \|A^{\gamma}u(t)\|_{L^2(\Omega)}
\leq C\int_0^t \left({t-\tau}\right)^{\bar{\alpha} -1}
                                            \|A^{\gamma}u(\tau)\|_{L^2(\Omega)}d\tau
     +C(\|F\|_{L^{\infty}(0,T;L^2(\Omega))}+\|a\|_{L^2(\Omega)} t^{-\alpha_1 \gamma}),
 \nonumber
 \end{align}
 where $\bar{\alpha}=\min(\alpha_1-\alpha_1 \gamma, \alpha_1-\alpha_2))=\alpha_1-\max(\alpha_1 \gamma, \alpha_2)$.
 
Moreover, from the general Gronwall inequality (e.g., Lemma 7.1.1 in \cite{Henry}), we see that
 \[
\|A^{\gamma}u(t)\|_{L^2(\Omega)} 
\leq C
\left(t^{-\alpha_1\gamma} \|a\|_{L^2(\Omega)}+\|F\|_{L^{\infty}(0,T;L^2(\Omega))}\right),
   0<t\leq T,
 \]
 where the constant $C>0$ only depend on $d$,  $\{\alpha_j\}_{j=1}^{\ell}$, $\gamma$, $T$, $B$, $\Omega$ and the coefficients of $\mathcal{A}$, $\{q_i\}_{i=2}^{\ell}$.
\end{proof}

Next we consider the unique existence of our fractional diffusion equation. On the basis of the fact that
the solution $u$ to the problem (\ref{multifrac}) satisfies the integral equation (\ref{solution aF}), 
we call the function $u$ which satisfies (\ref{solution aF}), $u\in C([0,T];L^2(\Omega))$ and 
$u(t)\in H_0^1(\Omega)$, a.e. $t\in [0,T]$ as the mild solution of the initial-boundary problem
(\ref{multifrac}). 

For any fixed $T>0$, 
we define the operator $\mathcal{K}$ as follows:
\begin{align}\label{Ku}
\mathcal{K}u(t) 
=\sum_{k=1}^5 I_k(t),
\end{align}
where $I_k(t)$, $k=1,\cdots,5$ are defined in (\ref{solution aF}). The homogeneous equation is firstly
investigated. Local existence result for mild solution was established via Banach's Fixed Point Theorem. Namely, the following theorem holds.

\begin{theorem}[Local existence] \label{Fixed point small T}
Suppose that $0<\alpha_{\ell}<\cdots<\alpha_1<1$, $a\in L^2(\Omega)$, $B(x):=(B_1(x),\cdots,B_{d}(x))$, $B_i\in W^{2,\infty}(\Omega)(i=1,\cdots,d)$, $q_j\in W^{2,\infty}(\Omega)(2\leq j \leq \ell)$, and $F=0$. Then there exists a mild solution to (\ref{doublefrac}) in the space 
$L^p(0,\delta;H_0^1(\Omega)) \cap C([0,\delta];L^2(\Omega))$ with $\gamma \in [\frac{1}{2},1)$, 
where $\delta>0$ is small enough, and $p\in(\frac{1}{\alpha_1},\frac{2}{\alpha_1})$.
\end{theorem}
In order to prove Theorem \ref{Fixed point small T}, we give some estimates, which are organized 
in the following lemmas.
\begin{lemma} \label{Ku L2 <}
Under the assumptions of Theorem \ref{Fixed point small T}. Then the following estimate
  \begin{align}
    \|\mathcal{K}u(t)\|_{L^2(\Omega)}
   \leq CT^{\alpha_1-\frac{1}{p}}\|u\|_{L^{p}(0,T;H_0^1(\Omega))}
       +C\sum_{i=2}^{\ell}T^{\alpha_1-\alpha_i}   \|u\|_{C([0,T];L^2(\Omega))}
       +C\|a\|_{L^2(\Omega)},\ t\in[0,T] \nonumber
  \end{align}
holds for each function $u\in L^{p}(0,T;H_0^1(\Omega))\cap C([0,T];L^2(\Omega))$, 
  $p\in(\frac{1}{\alpha_1},\frac{2}{\alpha_2})$.
\end{lemma}
Here and henceforth in this section, $C>0$ denotes constants only depend on $d$, 
$\{\alpha_j\}_{j=1}^{\ell}$, $\gamma$, $\Omega$ and the coefficients of the operator $\mathcal{A}$, $\{q_i\}_{i=2}^{\ell}$. Moreover, $C_T>0$ denotes constants only depending on $T$, $d$, $\{\alpha_j\}_{j=1}^{\ell}$, $\gamma$, $\Omega$ and the coefficients of the operator $\mathcal{A}$, $\{q_i\}_{i=2}^{\ell}$.
\begin{proof}
Using the assumption $u\in L^{p}(0,T;H_0^1(\Omega))\cap C([0,T];L^2(\Omega))$, we apply the same argument to
Theorem \ref{Estimate} and conclude that, for any $t\in[0,T]$,
\begin{align}
    \|\mathcal{K}u\|_{L^2(\Omega)}
\leq C\int_0^t (t-\tau)^{\alpha_1-1} \|u(\tau)\|_{H^1(\Omega)}d\tau
     +C\sum_{i=2}^{\ell}\int_0^t \tau^{\alpha_1-1-\alpha_i}d\tau \|u\|_{C([0,T];L^2(\Omega))}
     +C\|a\|_{L^2(\Omega)}.  \nonumber
\end{align}
Moreover, let $p'>0$ be the conjugate number of $p$, that is, $\frac{1}{p}+\frac{1}{p'}=1$, 
Since $p\in(\frac{1}{\alpha_1},\frac{2}{\alpha_1})$ implies 
\[
(\alpha_1-1)p'+1 = p'(\alpha_1-1 + \frac{1}{p'})=p'(\alpha_1-\frac{1}{p})>0,
\]
it follows from H\"{o}lder's inequality that, for $0\leq t\leq T$,
\begin{align*}
&\|\mathcal{K}u(t)\|_{L^2(\Omega)}\\
\leq&\ C\big(\int_0^t (t-\tau)^{(\alpha_1-1)p'}d\tau \big)^{\frac{1}{p'}}
       \big(\int_0^t \|u(\tau)\|_{H^1(\Omega)}^{p}d\tau\big)^{\frac{1}{p}}
     +C\sum_{i=2}^{\ell}t^{\alpha_1-\alpha_i} \|u\|_{C([0,T];L^2(\Omega))}
     +C\|a\|_{L^2(\Omega)}\\
\leq&\ CT^{\alpha_1-\frac{1}{p}}\|u\|_{L^p(0,T;H_0^1(\Omega))}
+C\sum_{i=2}^{\ell}T^{\alpha_1-\alpha_i}\|u\|_{C([0,T];L^2(\Omega))}+C\|a\|_{L^2(\Omega)}.
\end{align*}
\end{proof}

\begin{lemma} \label{Ku L1tilde <}
Under the assumptions of Theorem \ref{Fixed point small T}. Then for $\forall u\in L^{p}(0,T;H_0^1(\Omega))$ the following estimate holds:
\begin{align}
    \|\mathcal{K}u\|_{L^{p}(0,T;H_0^1(\Omega))}
\leq C\cdot(\sum_{i=2}^{\ell}T^{\alpha_1-\alpha_i} + T^{\frac{\alpha_1}{2}})
       \|u\|_{L^{p}(0,T;H_0^1(\Omega))}
     +CT^{\frac{1}{p}-\frac{\alpha_1}{2}} \|a\|_{L^2(\Omega)}.\nonumber
\end{align}
\end{lemma}
\begin{proof}
Similar to the calculation in Theorem \ref{Estimate}, we find
\begin{align*}
      \|A^{\frac{1}{2}}\mathcal{K}u(t)\|_{L^2(\Omega)}
\leq& C\int_0^t ((t-\tau)^{\frac{\alpha_1}{2}-1}
     +\sum_{i=2}^{\ell}(t-\tau)^{\alpha_1-\alpha_i-1})
\|A^{\frac{1}{2}}u(\tau)\|_{L^2(\Omega)}d\tau\\
     &+C(t^{-\frac{\alpha_1}{2}}
         +\sum_{i=2}^{\ell}t^{\frac{\alpha_1}{2}-\alpha_i})\|a\|_{L^2(\Omega)}.
\end{align*}
Therefore, since $|a+b|^p \leq 2^p(|a|^p+|b|^p)$, $\forall a,b\in\mathbb{R}$, we have
\begin{align*}
    \int_0^T \|\mathcal{K}u(t)\|_{D(A^{\frac{1}{2}})}^{p}dt 
\leq&\ C2^{p}\int_0^T\left(
      \int_0^t ((t-\tau)^{\frac{\alpha_1}{2}-1}
         +\sum_{i=2}^{\ell}(t-\tau)^{\alpha_1-\alpha_i-1})
      \|A^{\frac{1}{2}}u(\tau)\|_{L^2(\Omega)}d\tau
      \right)^{p} dt\\
     &+C2^{p}\int_0^T (t^{-\frac{\alpha_1}{2}}+\sum_{i=2}^{\ell}t^{\frac{\alpha_1}{2}-\alpha_i})^p\|a\|_{L^2(\Omega)}^{p}dt.
\end{align*}
By Young's inequality for the convolution, noting $p\in(\frac{1}{\alpha_1},\frac{2}{\alpha_1})$ implies
$\frac{\alpha_1}{2}p<1$, so that
\begin{align*}
\int_0^T \|\mathcal{K}u(t)\|_{D(A^{\frac{1}{2}})}^{p}dt
\leq& C\big(\int_0^T (\tau^{\frac{\alpha_1}{2}-1}
                     +\sum_{i=2}^{\ell}\tau^{\alpha_1-1-\alpha_i})d\tau\big)^{p}
      \int_0^T \|u(\tau)\|_{D(A^{\frac{1}{2}})}^{p}d\tau\\
    &+CT^{1-\frac{\alpha_1}{2}p} \|a\|_{L^2(\Omega)}^{p} 
    +C\sum_{i=2}^{\ell}T^{(\frac{\alpha_1}{2}-\alpha_i)p+1}\|a\|_{L^2(\Omega)}^p
\end{align*}
Finally, we obtain
\begin{align*}
\|\mathcal{K}u\|_{L^{p}(0,T;D(A^{\frac{1}{2}}))}
\leq& C(T^{\frac{\alpha_1}{2}}+\sum_{i=2}^{\ell}T^{\alpha_1-\alpha_j})\|u\|_{L^{p}(0,T;H_0^1(\Omega))}\\
    &+CT^{\frac{1}{p}-\frac{\alpha_1}{2}} \|a\|_{L^2(\Omega)}
    +C\sum_{i=2}^{\ell}T^{(\frac{\alpha_1}{2}-\alpha_i)p+1}\|a\|_{L^2(\Omega)}^p.
\end{align*}
 \end{proof}

\begin{proof}[\bf The proof of Theorem \ref{Fixed point small T}]
We set $X_T:=L^{p}(0,T;D(A^{\frac{1}{2}}))\cap C([0,T];L^2(\Omega))$ 
and $\|\cdot\|_T:=\|\cdot\|_{L^{p}(0,T;D(A^{\frac{1}{2}}))}+\|\cdot\|_{C([0,T];L^2(\Omega))}$.
It is easy to see that $\|\cdot\|_T$ is a norm of $X_T$, and $(X_T,\|\cdot\|_T)$ is a Banach space.

Assuming $u_1, u_2\in X_T$, by an argument similar to the proof of Lemma \ref{Ku L2 <} and 
Lemma \ref{Ku L1tilde <}, we derive that there exists a constant $C>0$ such that the following estimates
hold:
 \begin{align*} 
&\|\mathcal{K}u_1-\mathcal{K}u_2\|_{C([0,T];L^2(\Omega))}
 \leq CT^{\alpha_1-\frac{1}{p}} \|u_1-u_2\|_{L^{p}(0,T;H_0^1(\Omega))}
     +C\sum_{i=2}^{\ell}T^{\alpha_1-\alpha_i} \|u_1-u_2\|_{C([0,T];L^2(\Omega))}\\
&\|\mathcal{K}u_1-\mathcal{K}u_2\|_{L^{p}(0,T;H_0^1(\Omega))}
   \leq C(T^{\frac{\alpha_1}{2}} + \sum_{i=2}^{\ell}T^{\alpha_1-\alpha_i})
    \|u_1-u_2\|_{L^{p}(0,T;H_0^1(\Omega))}.
  \end{align*}

Setting $T=\delta$ , the above calculations leads to
\begin{align}
  \|\mathcal{K}u_1 - \mathcal{K}u_2\|_{C([0,\delta];L^2(\Omega))}
&\leq C[\delta^{\alpha_1-\frac{1}{p}} 
        + \sum_{i=2}^{\ell}\delta^{\alpha_1-\alpha_i}]\|u_1 - u_2\|_{X_{\delta}},
\label{constant C1}\\
  \|\mathcal{K}u_1 - \mathcal{K}u_2\|_{L^{p}(0,\delta;H_0^1(\Omega))}
&\leq C[\delta^{\frac{\alpha_1}{2}} + \sum_{i=2}^{\ell}\delta^{\alpha_1-\alpha_i}]
      \|u_1 - u_2\|_{L^{p}(0,\delta;H_0^1(\Omega))}.\label{constant C2}
\end{align}
From the above two estimates (\ref{constant C1}) and (\ref{constant C2}), it follows that the operator
$\mathcal{K}$ is a contracted operator from $(X_{\delta},\|\cdot\|_{\delta})$ into itself when $\delta>0$ 
is small enough. Consequently, there exists a unique fixed point $u\in X$ such that $\mathcal{K}u(t)=u(t)$, $\forall t\in[0,\delta]$.
\end{proof}

On the basis of Theorem \ref{Fixed point small T}, we can further prove the global existence of the mild solution to the initial-boundary value problem (\ref{multifrac}).
\begin{theorem}[Global existence] \label{Fixed point any T}
Under the assumption of Theorem \ref{Fixed point small T}. Then for any $T>0$ being fixed, there exists a mild solution
to (\ref{multifrac}) in $L^{p}(0,T;H_0^1(\Omega)) \cap C([0,T];L^2(\Omega))$, $p\in(\frac{1}{\alpha_1},\frac{2}{\alpha_1})$.
\end{theorem}
\begin{proof} 
For any fixed $T>0$, without loss of generality, we assume that $T>\delta$, where $\delta$ is defined
in Theorem \ref{Fixed point small T}. On the interval $[0,\delta]$, 
we see that $u\in X_{\delta}:= L^{p}(0,\delta;H_0^1(\Omega)) \cap C([0,\delta];L^2(\Omega))$ satisfies
(\ref{solution aF}). Here for convenience, we set
\[
u(t) = \sum_{i=2}^{\ell}\int_0^t A^{-1}   S'(t-\tau) q_i(x)\partial_t^{\alpha_i} u(\tau)d\tau
        -\int_0^t A^{-1} S'(t-\tau) B(x)\cdot \nabla u(\tau) d\tau
        +S(t)a,
\]
where 
$\sum_{i=2}^{\ell}\int_0^t A^{-1}S'(t-\tau) q_i(x)\partial_t^{\alpha_i}u(\tau)d\tau
:=I_2(t)+I_3(t)+I_4(t)$, 
and $I_2(t)$, $I_3(t)$, $I_4(t)$ are defined in (\ref{solution aF}). 
An argument similar to the proof of Theorem \ref{Estimate} leads to
 \begin{align}
    \|A^{\frac{1}{2}}u(t)\|_{L^2(\Omega)} 
\leq C_T\int_0^t \left({t-\tau}\right)^{\bar{\alpha} -1}
                                            \|A^{\frac{1}{2}}u(\tau)\|_{L^2(\Omega)}d\tau
     +C_T\|a\|_{L^2(\Omega)} t^{-\frac{\alpha_1}{2}},\ 0<t\leq \delta<T,
 \nonumber
 \end{align}
where $\bar{\alpha}=\min(\alpha_1-\alpha_1 \gamma, \alpha_1-\alpha_2)$, 
and the constant $C>0$ is independent of $\delta$.

We set $U(t)=\|A^{\frac{1}{2}}u(t)\|_{L^2(\Omega)}$ in $[0,\delta]$, and $U=0$ in the interval $(\delta,T]$. 
It is easy to see that $U(t)$ also satisfies the above inequality. Therefore by the general Gronwall's inequality ( e.g. Theorem 7.1.1 in \cite{Henry}), we have
\begin{align}
 U(t) \leq C_T t^{-\frac{\alpha_1}{2}} \|a\|_{L^2(\Omega)},
 \ t\in[0, T].\nonumber
\end{align}
Therefore
\begin{align} \label{Agamma Estimate on small T}
 \|u(t)\|_{D(A^{\frac{1}{2}})} \leq C_T t^{-\frac{\alpha_1}{2}} \|a\|_{L^2(\Omega)},
 \ t\in(0, \delta].
\end{align}
Hence by the integral equation of $u$ and the above inequality, similarly to Theorem \ref{Estimate}, 
we have
\begin{align}
   \|u(t)\|_{L^2(\Omega)}
\leq C_T(t^{\frac{\alpha_1}{2}} + 1)\|a\|_{L^2(\Omega)}
    +C_T\int_0^t \sum_{i=2}^{\ell}(t-\tau)^{\alpha_1-\alpha_i-1}\|u(\tau)\|_{L^2(\Omega)} d\tau,\ t\in[0,\delta]. \nonumber
\end{align}

Again from the general Gronwall's inequality, and similarly to (\ref{Agamma Estimate on small T}), 
we can prove that
\begin{align} \label{L^2 Estimate on small T}
  \|u(t)\|_{L^2(\Omega)} \leq C_T \|a\|_{L^2(\Omega)},\ t\in[0, \delta],
\end{align}
where $C_T>0$ depending only on $T$, $\{q_i\}_{i=2}^{\ell}$, $d$, $\{\alpha_j\}_{j=1}^{\ell}$,
$B$, $\Omega$ and the coefficients of the operator $\mathcal{A}$.

Next we study our initial-boundary problem on the interval $[\frac{\delta}{2},\delta]$. Denoting
$t_0:=\frac{\delta}{2}$, then $[\frac{\delta}{2},\delta]$ can be rewritten as
$[t_0,2t_0]$.

It is easy to see that the representation of the solution
\[
  u(t) = \sum_{i=2}^{\ell}\int_0^t A^{-1}S'(t-\tau) q_i(x)\partial_t^{\alpha_i}u(\tau)d\tau
        -\int_0^t A^{-1}S'(t-\tau) B(x)\cdot \nabla u(\tau)d\tau
        +S(t)a.
\]
still holds on interval $[t_0,2t_0]$. For $t\in [t_0,2t_0]$, we break up the integral 
into two parts
\begin{align}
  u(t)
=& \sum_{i=2}^{\ell}\int_{t_0}^t A^{-1}   S'(t-\tau) q_i(x)\partial_t^{\alpha_i} u(\tau)  d\tau
  -\int_{t_0}^t A^{-1}   S'(t-\tau) B(x)\cdot \nabla u(\tau)           d\tau
   \nonumber\\
 &+\sum_{i=2}^{\ell}\int_0^{t_0} A^{-1}   S'(t-\tau) q_i(x)\partial_t^{\alpha_i} u(\tau)  d\tau
  -\int_0^{t_0} A^{-1}   S'(t-\tau) B(x)\cdot \nabla u(\tau)           d\tau
          +S(t)a. \nonumber
\end{align}

We define a new operator $\mathcal{E}$ on interval $[t_0, 3t_0]$ by
\begin{align}
  \mathcal{E}v(t)
=& \sum_{i=2}^{\ell}\int_{t_0}^t A^{-1}   S'(t-\tau) q_i(x)\partial_t^{\alpha_i} v(\tau)  d\tau
  -\int_{t_0}^t A^{-1}   S'(t-\tau) B(x)\cdot \nabla v(\tau)           d\tau
   \nonumber\\
 &+\sum_{i=2}^{\ell}\int_0^{t_0} A^{-1}   S'(t-\tau) q_i(x)\partial_t^{\alpha_i} u(\tau)  d\tau
  -\int_0^{t_0} A^{-1}   S'(t-\tau) B(x)\cdot \nabla u(\tau)           d\tau
          +S(t)a \nonumber\\
=& \sum_{i=2}^{\ell}\int_{t_0}^t A^{-1}   S'(t-\tau) q_i(x)\partial_t^{\alpha_i} v(\tau)  d\tau
  -\int_{t_0}^t A^{-1}   S'(t-\tau) B(x)\cdot \nabla v(\tau)           d\tau
  +I(t).\nonumber
\end{align}

Now let us turn to the evaluation of $I(t)$, $t\in[t_0,3t_0]$. 
The use of (\ref{Agamma Estimate on small T}) and (\ref{L^2 Estimate on small T}) leads to
\begin{align}
  \|I(t)\|_{L^2(\Omega)}
\leq C_T
\left[\int_0^{t_0} (t-\tau)^{\alpha_1-1} \tau^{-\frac{\alpha_1}{2}} d\tau
     +\int_0^{t_0} \sum_{i=2}^{\ell}(t-\tau)^{\alpha_1-\alpha_i-1} d\tau
\right] \|a\|_{L^2(\Omega)}
    +C_T\|a\|_{L^2(\Omega)}. \nonumber
\end{align}
Therefore $\|I(t)\|_{L^2(\Omega)} \leq C_T\|a\|_{L^2(\Omega)}$, $t\in[t_0,3t_0]$. By an argument similar to the proof of Lemma \ref{Ku L2 <} and Lemma \ref{Ku L1tilde <}, 
we obtain that the operator $\mathcal{E}$ maps 
$X_1 :=C([t_0, 3t_0]; L^2(\Omega)) \cap L^{p}(t_0,3t_0;H_0^1(\Omega))$ 
into itself, where $p\in(\frac{1}{\alpha_1},\frac{2}{\alpha_1})$.

Let $v_1, v_2\in X_1$. By the definition of the operator $\mathcal{E}$, similarly to Theorem 
\ref{Fixed point small T}, we have
\begin{align*}
    &\|\mathcal{E}v_1(t) - \mathcal{E}v_2(t)\|_{L^2(\Omega)} \\
\leq&\ C\int_{t_0}^t (t-\tau)^{\alpha_1-1} \|v_1(\tau)-v_2(\tau)\|_{D(A^{\frac{1}{2}})}  d\tau
   + C\int_{t_0}^t \sum_{i=2}^{\ell}(t-\tau)^{\alpha_1-\alpha_i-1} d\tau \|v_1-v_2\|_{C([t_0,3t_0];L^2(\Omega))}.
\end{align*}
Consequently, the use of H\"{o}lder's inequality and $D(A^{\frac{1}{2}})=H_0^1(\Omega)$ yields that, for any
$t\in[t_0,3t_0]$, the following estimate
\begin{align*}
\|\mathcal{E}v_1(t) - \mathcal{E}v_2(t)\|_{L^2(\Omega)} 
\leq C\left((2t_0)^{\alpha_1-\frac{1}{p}}+ \sum_{i=2}^{\ell}(2t_0)^{\alpha_1-\alpha_i}\right)
      \|v_1-v_2\|_{X_1},\ p\in(\tfrac{1}{\alpha_1},\tfrac{2}{\alpha_1}).
\end{align*}
Moreover, similarly to the argument in Theorem \ref{Fixed point small T}, we can prove that
\begin{align*}
    &\|\mathcal{E}v_1(t) - \mathcal{E}v_2(t)\|_{D(A^{\frac{1}{2}})} \\
\leq&\ C\int_{t_0}^t (t-\tau)^{\frac{\alpha_1}{2}-1} \|v_1(\tau)-v_2(\tau)\|_{D(A^{\frac{1}{2}})}  d\tau
    +C\int_{t_0}^t \sum_{i=2}^{\ell}(t-\tau)^{\alpha_1-\alpha_i-1}  \|v_1(\tau)-v_2(\tau)\|_{D(A^{\frac{1}{2}})}  d\tau.
\end{align*}
Then by Young's inequality, for $p\in(\frac{1}{\alpha_1},\frac{2}{\alpha_1})$ we have
\begin{align*}
 \int_{t_0}^{3t_0} \|\mathcal{E}v_1(t) - \mathcal{E}v_2(t)\|_{D(A^{\frac{1}{2}})}^{p} dt
\leq C((2t_0)^{\frac{\alpha_1}{2}p}+\sum_{i=2}^{\ell}(2t_0)^{(\alpha_1-\alpha_i)p}) 
\|v_1-v_2\|_{L^{p}(t_0,3t_0;H_0^1(\Omega))}^p.
\end{align*}
By the notation $t_0=\frac{\delta}{2}$, we see that
\begin{align} 
    \|\mathcal{E}v_1(t) - \mathcal{E}v_2(t)\|_{L^2(\Omega)} 
\leq C\left(\delta^{\alpha_1-\frac{1}{p}}+ \sum_{i=2}^{\ell}\delta^{\alpha_1-\alpha_i}\right)
      \|v_1-v_2\|_{X_1},\ t\in[t_0,3t_0]
\label{Ev1-Ev2 L2 3t_0}\\
  \|\mathcal{E}v_1 - \mathcal{E}v_2\|_{L^p(t_0,3t_0;H_0^1(\Omega))}
\leq C\left(\sum_{i=2}^{\ell}\delta^{\alpha_1-\alpha_i}+\delta^{\frac{\alpha_1}{2}}\right)
      \|v_1-v_2\|_{L^{p}(t_0,3t_0;H_0^1(\Omega))}.\label{Ev1-Ev2 Agamma 3t_0}
\end{align}
From that the constant $C$ in (\ref{Ev1-Ev2 L2 3t_0}) and (\ref{Ev1-Ev2 Agamma 3t_0}) are same to the
constant in (\ref{constant C1}) and (\ref{constant C2}) and the choice of $\delta$ in 
Theorem \ref{Fixed point small T}, we can deduce that the operator $\mathcal{E}$ is strictly contracted 
operator from $(X_1,\|\cdot\|)$ into itself, where the norm $\|\cdot\|$ defined by
 \[
   \|u\| := \|u\|_{L^{p}(t_0,3t_0;H_0^1(\Omega))} + \|u\|_{C([t_0,3t_0];L^2(\Omega))}.
 \]
Therefore, there exists a unique fixed point $v\in X_1$ such that $\mathcal{E}v(t)=v(t)$ in $[t_0,3t_0]$,
that is,
 \begin{align} 
  v(t) = \sum_{i=2}^{\ell}\int_{t_0}^t A^{-1}S'(t-\tau)q_i(x)\partial_t^{\alpha_i} v(\tau)d\tau
        -\int_{t_0}^t A^{-1}   S'(t-\tau) B(x)\cdot \nabla v(\tau)           d\tau
        +I(t),\ t_0\leq t\leq3t_0.\nonumber
 \end{align}
Additionally, from the uniqueness argument we can see that $u(t)=v(t)$ in 
$[t_0, 2t_0]=[\frac{\delta}{2}, \delta]$, then we define a new function $\tilde{u}$ by
\begin{align}
  \tilde{u}=
\left\{
   \begin{array}{ll}
     u, & \hbox{$t\in[0,2t_0]$;} \\
     v, & \hbox{$t\in[2t_0, 3t_0]$.}
   \end{array}\right.\nonumber
\end{align}
Repeating the above argument to the interval pair $\big([0,3t_0], [2t_0,4t_0]\big)$, we can obtain that the mild solution exists on the larger interval $[0,4t_0]$, and go on. Finally the existence interval of the mild solution to our problem (\ref{doublefrac}) can be extended to the interval $[0,T]$, where $T>0$ is constant chosen at the very beginning.
\end{proof}
The discussion for the non-homogeneous equation with initial condition being zero is similar to the homogeneous equation. We list the results of homogeneous and non-homogeneous equations in the following theorem.
\begin{theorem} \label{regularity results}
Let $\{\alpha_j\}_{j=1}^{\ell}$ satisfy $0<\alpha_{\ell}<\cdots<\alpha_1<1$, 
and $\gamma\in[\frac{1}{2},1)$.
\begin{enumerate}
\item Let $a\in L^2(\Omega)$, $F=0$, then the initial-boundary value problem (\ref{doublefrac}) admits a unique mild
solution $u\in C((0,T];D(A^{\gamma}))\cap C([0,T];L^2(\Omega))$. 
Moreover, there exists a constant $C>0$ such that 
\[
\|u\|_{C([0,T];L^2(\Omega))} 
\leq C_T \|a\|_{L^2(\Omega)},
\]
and
\[
\|u(t)\|_{D(A^{\gamma})} \leq C_Tt^{-\alpha_1\gamma}\|a\|_{L^2(\Omega)},\ 0<t\leq T.
\]
\item Let $F\in L^{\infty}(0,T;L^2(\Omega))$, $a=0$, then the initial-boundary problem (\ref{doublefrac}) admits a unique solution $u\in C([0,T];D(A^{\gamma}))\cap L^2(0,T;H^2(\Omega))$. 
Moreover the following estimate holds:
\[
\|u\|_{C([0,T];D(A^{\gamma}))} + \|u\|_{L^2(0,T;H^2(\Omega))}
\leq C_T\|F\|_{L^{\infty}(0,T;L^2(\Omega))}.
\]
\end{enumerate}
\end{theorem}
\begin{proof}
(1). From Theorem \ref{Fixed point any T}, we see that there exists a unique mild solution in 
$L^{p}(0,T;H_0^1(\Omega)) \cap C([0,T];L^2(\Omega))$, $p\in(\frac{1}{\alpha_1},\frac{2}{\alpha_1})$, 
such that 
\[
u(t)=\sum_{i=2}^{\ell}\int_0^t A^{-1}S'(t-\tau)q_i(x)\partial_{\tau}^{\alpha_i}u(\tau)d\tau
 -\int_0^t A^{-1}S'(t-\tau)B(x)\cdot\nabla u(\tau)d\tau
+S(t)a.
\]
Therefore, by a similarly argument to Theorem \ref{Estimate}, we can prove that
\begin{align}
 \|A^{\frac{1}{2}}u(t)\|_{L^2(\Omega)}
\le C_T\int_0^t \left({t-\tau}\right)^{\bar{\alpha} -1}
       \|A^{\frac{1}{2}}u(\tau)\|_{L^2(\Omega)}d\tau
     +C_T\|a\|_{L^2(\Omega)} t^{-\frac{\alpha_1}{2}}, 0<t\leq T,
 \nonumber
 \end{align}
where $\bar{\alpha}=\alpha_1-\max(\alpha_1 \gamma, \alpha_2)$. By general Gronwall's inequality, 
the following estimate holds:
\begin{align} \label{A1/2u < a}
\|A^{\frac{1}{2}}u(t)\|_{L^2(\Omega)} \leq C_Tt^{-\frac{\alpha_1}{2}} \|a\|_{L^2(\Omega)}, 0<t\leq T.
\end{align}
By the density argument, we deduce that $u\in C((0,T];H_0^1(\Omega))$. For any arbitrary fixed small 
$0<\varepsilon<\frac{1}{2}$, taking $A^{\frac{1}{2}+\varepsilon}$ on the both sides of the above integral
equation of $u$, by (\ref{Ar-1s'}) and (\ref{Ar-1s''}) and (\ref{A1/2u < a}), similarly to the proof of Theorem \ref{Estimate}, we derive that
\begin{align}
&\|A^{\frac{1}{2}+\varepsilon} u(t)\|_{L^2(\Omega)}
\leq C\int_0^t (t-\tau)^{\frac{\alpha_1}{2}-\alpha_1\varepsilon-1}\|A^{\frac{1}{2}}u\|_{L^2(\Omega)}d\tau
+C\sum_{i=2}^{\ell}(t^{\frac{\alpha_1}{2}-\alpha_1\varepsilon-\alpha_i}+t^{-\alpha_1(\frac{1}{2}+\varepsilon)})
\|a\|_{L^2(\Omega)}\nonumber\\
&\qquad\qquad\qquad+C\int_0^t \sum_{i=2}^{\ell}(t-\tau)^{\alpha_1-\alpha_1\varepsilon-\alpha_i}
                              \|A^{\frac{1}{2}}u\|_{L^2(\Omega)}d\tau\nonumber\\
\leq&\ C_T\int_0^t (t-\tau)^{\frac{\alpha_1}{2}-\alpha_1\varepsilon-1}\tau^{-\frac{\alpha_1}{2}}\|a\|_{L^2(\Omega)}d\tau
+C_T(\sum_{i=2}^{\ell}t^{\frac{\alpha_1}{2}-\alpha_1\varepsilon-\alpha_i}
     +t^{-\frac{\alpha_1}{2}})\|a\|_{L^2(\Omega)}\nonumber\\
&+C_T\int_0^t \sum_{i=2}^{\ell}(t-\tau)^{\alpha_1-\alpha_1\varepsilon-\alpha_i}\tau^{-\frac{\alpha_1}{2}}
\|a\|_{L^2(\Omega)}d\tau.\nonumber\\
\leq&\ C_Tt^{-\alpha_1\varepsilon}\|a\|_{L^2(\Omega)}
+C_T(\sum_{i=2}^{\ell}t^{\frac{\alpha_1}{2}-\alpha_1\varepsilon-\alpha_i}
     +t^{-\alpha_1(\frac{1}{2}+\varepsilon})\|a\|_{L^2(\Omega)}
+C_T\sum_{i=2}^{\ell}t^{1+\frac{\alpha_1}{2}-\alpha_1\varepsilon-\alpha_i}
\|a\|_{L^2(\Omega)},\nonumber
\end{align}
which yields that $u\in C((0,T];D(A^{\frac{1}{2}+\varepsilon}))$ and 
\begin{align*}
\|A^{\frac{1}{2}+\varepsilon} u(t)\|_{L^2(\Omega)}
\leq C_Tt^{-\widetilde{\alpha}}\|a\|_{L^2(\Omega)},\ 0<\widetilde{\alpha}<1,\ 0< t\leq T.
\end{align*}
Repeating the above argument, we deduce $u\in C((0,T];D(A^{\frac{1}{2}+2\varepsilon}))$ and 
\begin{align*}
\|A^{\frac{1}{2}+2\varepsilon} u(t)\|_{L^2(\Omega)}
\leq C_Tt^{-\tilde{\alpha}_1}\|a\|_{L^2(\Omega)},\ 0<\widetilde{\alpha}_1<1,\ 0< t\leq T.
\end{align*}
Consequently, step by step, we obtain that for any $\gamma\in[\frac{1}{2},1)$ there exists a constant $C>0$ depending only on 
$d$, $\{q_i\}_{i=2}^{\ell}$, $\{\alpha_j\}_{j=1}^{\ell}$, $\gamma$, $B$, $\Omega$, $T$ 
and the coefficients of the operator $\mathcal{A}$, such that $u\in C((0,T];D(A^{\gamma}))$ and
$
\|A^{\gamma} u(t)\|_{L^2(\Omega)}
\leq Ct^{-\alpha_1\gamma}\|a\|_{L^2(\Omega)}
$,
$t\in(0,T]$.

(2). Firstly, we show the existence of the mild solution of initial-boundary value problem
(\ref{doublefrac}). It is easy to show that the operator $\mathcal{K}$ defined by (\ref{solution aF}) maps
the space $X_2$ into itself, where \[X_2:=\{u\in C([0,T];D(A^{\gamma}));u(0)=0\}\] with the norm 
$\|\cdot\|_{C([0,T];D(A^{\gamma}))}$. Moreover, by induction, we can obtain that for any $u$, $v$ in $X_2$,
the following estimation 
\[
\|\mathcal{K}^nu(t)-\mathcal{K}^nv(t)\|_{D(A^{\gamma})}
\leq \frac{M_1^nt^{\bar{\alpha}n}}{\Gamma(\bar{\alpha}n+1)}\|u-v\|_{C([0,T];D(A^{\gamma}))},\ 0< t\leq T
\]
holds, where $\bar{\alpha}=\alpha_1-\max(\alpha_1\gamma,\alpha_2)$. In fact, by the definition of the operator $\mathcal{K}$ and the induct assumption, we have
\begin{align*}
&\|\mathcal{K}^{n+1}u(t)-\mathcal{K}^{n+1}v(t)\|_{D(A^{\gamma})} 
\leq C\int_0^t (t-\tau)^{\bar{\alpha}-1} 
 \frac{M_1^n{\tau}^{\bar{\alpha}n}}{\Gamma(\bar{\alpha}n+1)}\|u-v\|_{C([0,T];D(A^{\gamma}))}d\tau,
 \ 0\leq t\leq T.
\end{align*}
Using $B(\alpha,\beta):=\frac{\Gamma(\alpha)\Gamma(\beta)}{\Gamma(\alpha+\beta)}
=\int_0^1 s^{\alpha-1}(1-s)^{\beta-1}ds$, $\alpha,\beta>0$, setting $M_1:=C\Gamma(\bar{\alpha})$, 
we see that
\begin{align*}
\|\mathcal{K}^{n+1}u(t)-\mathcal{K}^{n+1}v(t)\|_{D(A^{\gamma})}
\leq&\ \frac{CM_1^n}{\Gamma(\bar{\alpha}n+1)}B(\bar{\alpha},\bar{\alpha}n+1) t^{\bar{\alpha}(n+1)}
\|u-v\|_{C([0,T];D(A^{\gamma}))} \\
=&\ \frac{M_1^{n+1}t^{\bar{\alpha}n}}{\Gamma(\bar{\alpha}{n+1}+1)}\|u-v\|_{C([0,T];D(A^{\gamma}))},\ 0\leq t\leq T.
\end{align*}
Consequently, for $n\in\mathbb{N}$ big enough, we can see that $\mathcal{K}$ is a strictly contracted operator from $X_2$ to $X_2$. Then there exists a unique fixed point $\bar{u}\in X_2$ such that $\mathcal{K}^n\bar{u}=\bar{u}$.
It is easy to show that $\bar{u}$ is also the fixed point of the operator $\mathcal{K}$, that is, 
$\mathcal{K}\bar{u}=\bar{u}$, and therefore the fixed point of $\mathcal{K}:X_2\rightarrow X_2$ is also unique. 

Similar to the proof of Theorem \ref{Estimate}, we can prove 
   $u\in C([0,T];D(A^{\gamma}))$ and the estimation 
\begin{align} \label{u C[0,T];D(Agamma) leq F}
\|u\|_{C([0,T];D(A^{\gamma}))}\leq C\|F\|_{L^{\infty}(0,T;L^2(\Omega))}.   
\end{align}
The use of $\|\nabla u\|_{L^2(\Omega)}\leq C\|F\|_{L^{\infty}(0,T); L^2(\Omega)}$ and Theorem 2.1 in \cite{SakYam} leads to
 \[
 I_1(t):=\int_0^t A^{-1} S'(t-\tau) (B\cdot\nabla u + F)d\tau \in L^2(0,T;H^2(\Omega)),
 \]
 and
 \begin{align}\label{v_1}
\|I_1\|_{L^2(0,T;H^2(\Omega))} \leq C\|F\|_{L^{\infty}(0,T; L^2(\Omega))}.
 \end{align}
Next we show that
$\sum_{i=2}^{\ell}\int_0^t A^{-1} S'(t-\tau) q_i\partial_{\tau}^{\alpha_i}u d\tau:=I_2(t)+I_3(t) \in
C([0,T];H^2(\Omega))$, where $I_2(t)$, $I_3(t)$ are defined in (\ref{solution aF}).

For any $\gamma\in[\frac{1}{2},1)$, $\varepsilon_0>0$ small enough such that 
$\alpha_1-\alpha_i-\alpha_1\varepsilon_0 > 0$ $(i=2,\cdots,\ell)$, from (\ref{u C[0,T];D(Agamma) leq F}), 
similarly to Theorem \ref{Estimate}, we see that
\begin{align*}
&\|A^{\gamma+\varepsilon_0}(I_2+I_3)(t)\|_{L^2(\Omega)}
\leq C\int_0^t \sum_{i=2}^{\ell}(t-\tau)^{\alpha_1-\alpha_i-\alpha_1\varepsilon_0-1}\|A^{\gamma}u\|_{L^2(\Omega)} d\tau \\
\leq&\ C\int_0^t \sum_{i=2}^{\ell}(t-\tau)^{\alpha_1-\alpha_i-\alpha_1\varepsilon_0-1} d\tau 
\|F\|_{L^{\infty}(0,T;L^2(\Omega))}
\leq C \|F\|_{L^{\infty}(0,T;L^2(\Omega))}.
\end{align*}
 Since $\gamma\in[\frac{1}{2},1)$, we can choose $\gamma$ such that $\gamma+\varepsilon_0=1$, then we have
 \begin{align}\label{v_2}
 \|A(I_2+I_3)(t)\|_{L^2(\Omega)}
\leq C \|F\|_{L^{\infty}(0,T;L^2(\Omega))},\ 0\leq t\leq T.
 \end{align}
Hence $I_2+I_3\in C([0,T];H^2(\Omega))$. Collecting the above estimates (\ref{v_1}) and (\ref{v_2}), we derive $u\in L^2(0,T;H^2(\Omega))$, 
 and 
 \[\|u\|_{L^2(0,T;H^2(\Omega))} \leq C\|F\|_{L^{\infty}(0,T; L^2(\Omega))}.\]
\end{proof}


\section{H\"{o}lder regularity}

Here we assume that $a=0$, $F\in C^{\theta}([0,T];L^2(\Omega))$ with $\theta\in(0,1)$ and $F(0)=0$. 
We expect to obtain some H\"{o}lder estimate
for the mild solution to the equation (\ref{doublefrac}).

We set $X_{\theta}:=\{Au\in C^{\theta}([0,T];L^2(\Omega)), u(0)=0\}$, with H\"{o}lder norm 
\[
\|Au\|_{\theta}:=\|Au\|_{C^0([0,T];L^2(\Omega))}
 + \sup_{t_1\neq t_2\in [0,T]}\frac{\|Au(t_1)-Au(t_2)\|_{L^2(\Omega)}}{|t_1-t_2|^{\theta}}.
\]

From the arguments in the above section, we can formally obtain that $u(t)=\mathcal{K}u(t)$, $t\in[0,T]$,
where the operator $\mathcal{K}$ is defined by (\ref{Ku}). Firstly, we want to prove that the operator $\mathcal{K}$ can improve the H\"{o}lder's regularity. Namely, the following lemma holds.

\begin{lemma}\label{Ku(t+h)-Ku(t)}
There exist constant $\varepsilon>0$ small enough, and constant $C>0$, for any $\theta'\in[0,\theta)$, 
the following estimate
\[
\|\mathcal{K}u(t+h)-\mathcal{K}u(t)\|_{L^2(\Omega)} 
\leq C(h^{\theta}\|F\|_{\theta}+h^{\theta'+\varepsilon}\|Au\|_{\theta'})
\]
 holds any $u\in X_{\theta'}$.
\end{lemma}
\begin{proof}
We recall the definition of $I_j$, $j=1,\cdots,5$ in (\ref{solution aF}).
Let $h>0$, and $t,t+h\in[0,T]$. Representing $AI_1$ as
\begin{align*}
AI_1(t)
=\int_0^tS'(\tau)F(t-\tau)d\tau + \int_0^tS'(\tau)B\cdot\nabla u(t-\tau)d\tau=:I_{11}(t)+I_{12}(t).
\end{align*}
By an argument similar to the proof of Theorem 2.4 in \cite{SakYam}, we can prove that
\begin{align*}
\|I_{11}(t+h)-I_{11}(t)\|_{L^2(\Omega)}
\leq Ch^{\theta}\|F\|_{\theta}. 
\end{align*}

For $0\leq t<t+h\leq T$, $h<1$, we have
\begin{align}
&I_{12}(t+h)-I_{12}(t)
=\int_{-h}^{t} S'(\tau+h)(B\cdot\nabla u(t-\tau))d\tau
-\int_0^t S'(\tau)(B\cdot\nabla u(t-\tau))d\tau \nonumber\\
=&\int_{-h}^0 A^{-\frac{1}{2}}S'(\tau+h)A^{\frac{1}{2}}(B\cdot\nabla u(t-\tau)-B\cdot\nabla u(t))d\tau
+(S(t+h)-S(t))(B\cdot\nabla u(t))
\nonumber\\
&+\int_0^t A^{-\frac{1}{2}}(S'(\tau+h)-S'(\tau))A^{\frac{1}{2}}(B\cdot\nabla u(t-\tau)-B\cdot\nabla u(t))d\tau
=:J_1^h(t)+J_2^h(t)+J_3^h(t). \nonumber
\end{align}
Using the estimate (\ref{Ar-1s'}) and (\ref{Ar-1s''}), we find 
\begin{align}
\|J_1^h(t)\|_{L^2(\Omega)}\leq C\int_{-h}^0 (\tau+h)^{\frac{\alpha_1}{2}-1} \tau^{\theta'}\|Au\|_{\theta'}
\leq Ch^{\frac{\alpha_1}{2}+\theta'}\|Au\|_{\theta'}. \nonumber
\end{align}
and noting that $\int_{\tau}^{\tau+h}S''(\xi)d\xi=S'(\tau+h)-S'(\tau)$, and again using (\ref{Ar-1s''}), 
we can prove that
\begin{align}\label{J_2}
&\|J_3^h(t)\|_{L^2(\Omega)}
\leq Ch^{\frac{\alpha_1}{2}+\theta'}
  \int_0^{\frac{t}{h}} (\tau^{\frac{\alpha_1}{2}-1}-(\tau+1)^{\frac{\alpha_1}{2}-1})\tau^{\theta'} d\tau 
  \|Au\|_{\theta'}.
\end{align}
In the case when $0\leq t \leq h$, according to (\ref{J_2}), we derive the estimate
\begin{align}\label{J_2t<h}
\|J_3^h(t)\|_{L^2(\Omega)}
\leq Ch^{\frac{\alpha_1}{2}+\theta'}
  \int_0^1 (\tau^{\frac{\alpha_1}{2}-1}-(\tau+1)^{\frac{\alpha_1}{2}-1})\tau^{\theta'} d\tau\|Au\|_{\theta'}
\leq Ch^{\frac{\alpha_1}{2}+\theta'} \|Au\|_{\theta'}.
\end{align}
As for the case $t>h$, we represent $J_3^h(t)$ as
\[
\|J_3^h(t)\|_{L^2(\Omega)}
\leq  Ch^{\frac{\alpha_1}{2}+\theta'}\left(\int_0^1 + \int_0^{\frac{t}{h}}\right)
 (\tau^{\frac{\alpha_1}{2}-1}-(\tau+1)^{\frac{\alpha_1}{2}-1})\tau^{\theta'} d\tau\|Au\|_{\theta'},
\] 
Therefore, since (\ref{J_2t<h}) and the inequality 
$\tau^{\frac{\alpha_1}{2}-1}-(\tau+1)^{\frac{\alpha_1}{2}-1} \leq C\tau^{\frac{\alpha_1}{2}-2}$,\ $\tau>1$,
it follows that
\begin{align*}
\|J_3^h(t)\|_{L^2(\Omega)}
\leq Ch^{\frac{\alpha_1}{2}+\theta'} \|Au\|_{\theta'}
     +Ch^{\frac{\alpha_1}{2}+\theta'}
       \int_1^{\infty} \tau^{\frac{\alpha_1}{2}-2}\tau^{\theta'} d\tau\|Au\|_{\theta'}
\leq Ch^{\frac{\alpha_1}{2}+\theta'} \|Au\|_{\theta'},
\end{align*}

We estimate $J_2^h(t)$. 
Since $u(0)=0$, it follows that 
$\|A^{\frac{1}{2}}B\cdot \nabla u(t)\|_{L^2(\Omega)}\leq Ct^{\theta'}\|Au\|_{\theta'}$, 
we can prove that
\begin{align*}
&\|J_3^h(t)\|_{L^2(\Omega)}
=\left\| \int_t^{t+h} A^{-\frac{1}{2}} S'(\eta)d\eta A^{\frac{1}{2}}B\cdot\nabla u(t)d\eta \right\|
_{L^2(\Omega)}
\leq Ch^{\frac{\alpha_1}{2}+\theta'}\|Au\|_{\theta'}.
\end{align*}

Therefore, we proved that
$
\|I_{12}(t+h)-I_{12}(t)\|_{L^2(\Omega)}
\leq Ch^{\theta'+\frac{\alpha_1}{2}}\|Au\|_{\theta'},\ 0\leq t<t+h\leq T.
$
We are to estimate $AI_2(t)$. Similarly to the calculation of $AI_1(t)$, 
for any $0\leq t<t+h\leq T$, we have
\begin{align*}
(AI_2&(t+h)-AI_2(t))
=\sum_{i=2}^{\ell}\frac{1}{\Gamma(1-\alpha_i)}
  \int_{-h}^0 A^{-1} S'(\tau+h) (\tau+h)^{-\alpha_i} A(q_iu(t-\tau)-qu(t))d\tau \\
&+\sum_{i=2}^{\ell}\frac{1}{\Gamma(1-\alpha_i)}
  \int_0^{t}   A^{-1} (S'(\tau+h)(\tau+h)^{-\alpha_i}-S'(\tau)\tau^{-\alpha_i}) A(q_iu(t-\tau)-qu(t))d\tau \\
&+\sum_{i=2}^{\ell}\frac{1}{\Gamma(1-\alpha_i)}
  \int_t^{t+h} A^{-1} S'(\tau)\tau^{-\alpha_i} A(q_iu(t))d\tau = I_{21}^h(t)+I_{22}^h(t)+I_{23}^h(t).
\end{align*}
Therefore, in terms of estimate $\|A^{-1}S'(t)\|\leq Ct^{\alpha_1-1}$, 
$\forall t>0$ and that $u\in X_{\theta'}$, we obtain
\begin{align*}
\|I_{21}^h(t)\|_{L^2(\Omega)}
\leq C \int_{-h}^0 \sum_{i=2}^{\ell}(\tau+h)^{\alpha_1-\alpha_i-1}|\tau|^{\theta'} \|Au\|_{\theta'}
\leq C\sum_{i=2}^{\ell}h^{\alpha_1-\alpha_i+\theta'} \|Au\|_{\theta'},\ 
0\leq t<t+h\leq T,
\end{align*}
and noting that 
$\int_{\tau}^{\tau+h} d(S'(\eta)\eta^{-\alpha_i})=S'(\tau+h)(\tau+h)^{-\alpha_i}-S'(\tau)\tau^{-\alpha_i}$,
$i=2,\cdots,\ell$. Therefore for $\forall \theta''\in[0,\theta']$, 
again using (\ref{Ar-1s'}) and (\ref{Ar-1s''}) we have
\begin{align*}
\|I_{22}^h(t)\|_{L^2(\Omega)}
\leq&\ C\sum_{i=2}^{\ell}\left\|
\int_0^{t}\int_{\tau}^{\tau+h} A^{-1}S''(\eta)\eta^{-\alpha_i}d\eta 
A(q_iu(t-\tau)-q_iu(t))d\tau
\right\|_{L^2(\Omega)} \\
&+C\sum_{i=2}^{\ell}\left\|
\int_0^{t}\int_{\tau}^{\tau+h} A^{-1}\alpha_iS'(\eta)\eta^{-\alpha_i-1}d\eta 
A(q_iu(t-\tau)-q_iu(t))d\tau
\right\|_{L^2(\Omega)} \\
\leq&\ C\sum_{i=2}^{\ell}h^{\alpha_1-\alpha_i+\theta''}
\int_0^{\frac{t}{h}} (\tau^{\alpha_1-\alpha_i-1}-(\tau+1)^{\alpha_1-\alpha_i-1}) \tau^{\theta''}d\tau
 \|Au\|_{\theta''},\ 0\leq t<t+h\leq T.
\end{align*}
If $0\leq t\leq h$, then we have 
\[
\|I_{22}^h(t)\|_{L^2(\Omega)} 
\leq C\sum_{i=2}^{\ell}h^{\alpha_1-\alpha_i+\theta''}\|Au\|_{\theta''},\ 0\leq t<t+h\leq T.
\]
If $t>h$, we choose $\varepsilon>0$, $\theta''\in[0,\theta']$ such that $\alpha_1-\alpha_{\ell}+\theta''<1$, 
and $\alpha_1-\alpha_2+\theta''=\theta'+\varepsilon$. 
In fact, we can choose $\theta'':=c_0\min\{1-\alpha_1+\alpha_{\ell},\theta'\}$, 
$\varepsilon:=\alpha_1-\alpha_2+\theta''-\theta'$, here $0<c_0<1$ is sufficiently close to $1$.
Using the fact that $C^{\alpha}([0,T];L^2(\Omega))\subset C^{\beta}([0,T];L^2(\Omega))$, where $0\leq\beta\leq\alpha\leq 1$, we obtain that for 
$0<h<1$ the following estimate
\begin{align*}
&\|I_{22}^h(t)\|_{L^2(\Omega)} 
\leq C\sum_{i=2}^{\ell}h^{\alpha_1-\alpha_i+\theta''}
   \left(
         1+\sum_{i=2}^{\ell}\int_1^{\infty} 
                             (\tau^{\alpha_1-\alpha_i-1}-(\tau+1)^{\alpha_1-\alpha_i-1})\tau^{\theta''} d\tau
   \right)\|Au\|_{\theta''} \\
\leq&\ C\sum_{i=2}^{\ell}h^{\alpha_1-\alpha_i+\theta''}\|Au\|_{\theta''}
   +C\sum_{i=2}^{\ell}h^{\alpha_1-\alpha_i+\theta''}\int_1^{\infty} \tau^{\alpha_1-\alpha_i-2}\tau^{\theta''} d\tau
   \|Au\|_{\theta''}
\leq Ch^{\theta'+\varepsilon}\|Au\|_{\theta'},
\end{align*}
holds. Choosing $\varepsilon>0$, $\theta''\in[0,\theta']$ as in the calculation of $I_{22}^h(t)$, then we have
\begin{align*}
&\|I_{23}^h(t)\|_{L^2(\Omega)}
\leq C\int_t^{t+h} \sum_{i=2}^{\ell}\tau^{\alpha_1-\alpha_i-1}t^{\theta''}d\tau\|Au\|_{\theta''}
\leq C\int_t^{t+h} \sum_{i=2}^{\ell}\tau^{\alpha_1-\alpha_i-1}\tau^{\theta''}d\tau\|Au\|_{\theta''}.
\end{align*}
Consequently, we have
\begin{align*}
\|AI_2(t+h)-AI_2(t)\|_{L^2(\Omega)}
\leq Ch^{\theta'+\varepsilon}\|Au\|_{\theta'}, 0\leq t<t+h\leq T, 0<h<1.
\end{align*}

It remains to estimate $\|AI_3(t+h)-AI_3(t)\|_{L^2(\Omega)}$.
For brevity, for $i=2,\cdots,\ell$ we set
\[
g_i(\tau)
:=\frac{1}{\Gamma(1-\alpha_i)}
   \int_0^{\tau} A^{-1}S''(\tau-\eta)(\eta^{-\alpha_i}-\tau^{-\alpha_i})d\eta.
\]
For $0\leq t<t+h\leq T$, similarly to the above calculation, we have
\begin{align}
&AI_3(t+h)-AI_3(t)\nonumber\\
=&\sum_{i=2}^{\ell}\int_{-h}^0 g_i(\tau+h) A(q_iu(t-\tau)-q_iu(t))d\tau
 +\sum_{i=2}^{\ell}\int_0^t (g_i(\tau+h)-g_i(\tau)) A(q_iu(t-\tau)-q_iu(t))d\tau \nonumber\\
 &+\sum_{i=2}^{\ell}\int_t^{t+h} g_i(\tau) A(q_iu(t))d\tau
 =:I_{31}^h(t)+I_{32}^h(t)+I_{33}^h(t),\ 0\leq t<t+h\leq T.
 \nonumber
\end{align}
Again applying the estimate (\ref{Ar-1s'}) and (\ref{Ar-1s''}), we obtain, for any $0\leq t<t+h\leq T$
\begin{align*}
\|I_{31}^h(t)\|_{L^2(\Omega)}
\leq C\sum_{i=2}^{\ell}
       \int_{-h}^0 \int_0^{\tau+h} (\tau+h-\eta)^{\alpha_1-2} (\eta^{-\alpha_i}-(\tau+h)^{-\alpha_i})d\eta
                                     |\tau|^{\theta'}d\tau \|Au\|_{\theta'}.
\end{align*}
By changing variable $\tilde{\eta}=\frac{\eta}{\tau+h}$, we find, 
for any $0\leq t<t+h\leq T$, $0<h<1$
\begin{align*}
\|I_{31}^h(t)\|_{L^2(\Omega)}
\leq&\ C\sum_{i=2}^{\ell}\int_{-h}^0 \int_0^1 (1-\eta)^{\alpha_1-2}
 (\eta^{-\alpha_i}-1)d\eta |\tau|^{\theta'}(\tau+h)^{\alpha_1-\alpha_i-1}d\tau 
 \|Au\|_{\theta'}\\
\leq& Ch^{\alpha_1-\alpha_2+\theta''}\|Au\|_{\theta''}
=Ch^{\theta'+\varepsilon}\|Au\|_{\theta'}.
\end{align*}
Choosing $\varepsilon>0$, $\theta''\in[0,\theta']$ as in the calculation of $I_{22}^h(t)$, 
we have for $0\leq t<t+h\leq T$
\begin{align}
&\|I_{33}^h(t)\|_{L^2(\Omega)}
\leq C\sum_{i=2}^{\ell}
       \int_t^{t+h}\int_0^{\tau}(\tau-\eta)^{\alpha_1-2}(\eta^{-\alpha_i}-\tau^{-\alpha_i})d\eta t^{\theta''}d\tau
         \|Au\|_{\theta''}\nonumber\\
=&\ C\sum_{i=2}^{\ell}
      \int_t^{t+h}\int_0^1(1-\eta)^{\alpha_1-2}(\eta^{-\alpha_i}-1)d\eta\tau^{\alpha_1-\alpha_i-1}t^{\theta''}d\tau
        \|Au\|_{\theta''}
\leq Ch^{\theta'+\varepsilon}\|Au\|_{\theta'}, 0<h<1.\nonumber
\end{align}
In order to estimate $\|I_{32}^h(t)\|_{L^2(\Omega)}$, we need to give an estimate of 
$\|g_i(\tau +h)-g_i(\tau )\|_{L^2(\Omega)}$, $i=2,\cdots,\ell$.
\begin{align}
{\Gamma(1-\alpha_i)}&(g_i(\tau+h)-g_i(\tau))
=\int_{-h}^0 A^{-1} S''(\tau-\eta)((\eta+h)^{-\alpha_i}-(\tau+h)^{-\alpha_i})d\eta\nonumber\\
 &+\int_0^{\tau} A^{-1} S''(\tau-\eta)
  \left((\eta+h)^{-\alpha_i}-(\tau+h)^{-\alpha_i}-\eta^{-\alpha_i}+\tau^{-\alpha_i}\right)d\eta
=:J_{4}^{hi}(\tau)+J_{5}^{hi}(\tau).\nonumber
\end{align}
The case $t\leq h$ is considered firstly. To estimate $J_{4}^{hi}(\tau)$ 
$0\leq \tau\leq t\leq h$, we represent it in the form
\begin{align*}
&\|J_{4}^{hi}(\tau)\|
\leq C\left(\int_{-h}^{-\frac{h}{2}} + \int_{-\frac{h}{2}}^0\right) (\tau-\eta)^{\alpha_1-2}((\eta+h)^{-\alpha_i}-(\tau+h)^{-\alpha_i})d\eta := J_{41}^{hi}+J_{42}^{hi}.
\end{align*}
For $J_{41}^{hi}$. Using the inequality $(\tau-\eta)^{\alpha_1-2}\leq (\frac{h}{2})^{\alpha_1-2}$, 
$\forall \tau\in[0,h]$,
$\forall \eta\in[-h,-\frac{h}{2}]$, we have
\begin{align}
\|J_{41}^{hi}(\tau)\|
\leq C\int_{-h}^{-\frac{h}{2}} h^{\alpha_1-2} ((\eta+h)^{-\alpha_i}-(\tau+h)^{-\alpha_i})d\eta
\leq Ch^{\alpha_1-\alpha_i-1}.\nonumber
\end{align}
For $J_{42}^{hi}$. Noting that 
$(\eta+h)^{-\alpha_i}-(\tau+h)^{-\alpha_i}\leq C(\tau-\eta)(\eta+h)^{-\alpha_i-1}$, $\forall\tau\in[0,h]$, $\forall\eta\in[-\frac{h}{2},0]$,
we have
\begin{align*}
\|J_{42}^{hi}(\tau)\|
\leq C\int_{-\frac{h}{2}}^0 (\tau-\eta)^{\alpha_1-1}(\eta+h)^{-\alpha_i-1}d\eta
\leq C\tau^{\alpha_1-1}h^{-\alpha_i}.
\end{align*}
Finally we deduce
\[
\|J_{4}^{hi}(\tau)\|\leq Ch^{\alpha_1-\alpha_i-1} +Ch^{-\alpha_i}\tau^{\alpha_1-1}, \forall 0<\tau\leq h.
\]
It remains to estimate $J_{5}^{hi}(\tau)(0< \tau\leq h)$. From (\ref{Ar-1s''}), we see that
\begin{align}
&\|J_{5}^{hi}(\tau)\|
\leq C\int_0^{\tau} (\tau-\eta)^{\alpha_1-2}
  \left|(\eta+h)^{-\alpha_i}-(\tau+h)^{-\alpha_i}-\eta^{-\alpha_i}+\tau^{-\alpha_i}\right|d\eta\nonumber\\
\leq&\ C\int_0^{\frac{\tau}{2}} (\tau-\eta)^{\alpha_1+\alpha_i-2}
  \left|(\eta+h)^{-\alpha_i}(\tau+h)^{-\alpha_i}+\eta^{-\alpha_i}\tau^{-\alpha_i}\right|d\eta\nonumber\\
  &+C\int_{\frac{\tau}{2}}^{\tau} (\tau-\eta)^{\alpha_1-2}
  \left((\eta+h)^{-\alpha_i}-(\tau+h)^{-\alpha_i}-\eta^{-\alpha_i}+\tau^{-\alpha_i}\right)d\eta
=:J_{51}^{hi}+J_{52}^{hi}, 0< \tau \leq h.\nonumber
\end{align}
For $J_{51}^{hi}$. Using the inequality 
$(\tau-\eta)^{\alpha_1+\alpha_i-2}\leq (\frac{\tau}{2})^{\alpha_1+\alpha_i-2}$, 
$\forall \tau\geq 0$,
$\forall \eta\in[0,\frac{\tau}{2}]$, we have
\begin{align}
\|J_{51}^{hi}(\tau)\|
\leq& C\int_0^{\frac{\tau}{2}} \tau^{\alpha_1+\alpha_i-2}(h^{-2\alpha_i}+ \eta^{-\alpha_i}\tau^{-\alpha_i})d\eta
\leq Ch^{-2\alpha_i}\tau^{\alpha_1+\alpha_i-1} + C\tau^{\alpha_1-\alpha_i-1},\ 0<\tau\leq h.
\nonumber
\end{align}
For $J_{52}^{hi}$. Using the inequality 
\begin{align}
(\eta+h)^{-\alpha_i}-(\tau+h)^{-\alpha_i} &\leq C(\tau-\eta)(\eta+h)^{-\alpha_i-1}, 0\leq \eta\leq \tau,\nonumber\\
\eta^{-\alpha_i}-\tau^{-\alpha_i} &\leq C(\tau-\eta)\eta^{-\alpha_i-1}, 0<\eta\leq \tau ,\nonumber
\end{align}
 we have
 \begin{align}
 \|J_{52}^{hi}(\tau)\|
\leq C\int_{\frac{\tau}{2}}^{\tau} (\tau-\eta)^{\alpha_1-1} \eta^{-\alpha_i-1}d\eta
\leq C\tau^{\alpha_1-\alpha_i-1},\ 0<\tau\leq h.\nonumber
 \end{align}

Collecting the estimates for $J_{4}^{hi}(\tau)$ and $J_{5}^{hi}(\tau)$  
on $0< \tau\leq t\leq h$ we obtain 
\begin{align} \label{g <h}
\|g_i(\tau +h)-g_i(\tau)\|
\leq C\tau^{\alpha_1-\alpha_i-1},\ 0< \tau\leq t\leq h, i=2,\cdots,\ell.
\end{align}
We choose $\varepsilon>0$, $\theta''\in[0,\theta']$ as in the calculation of $I_{22}^h(t)$. Then (\ref{g <h}) yields the estimate
\begin{align}
\|I_{32}^h(t)\|_{L^2(\Omega)}
\leq C\sum_{i=2}^{\ell}\int_0^t \|g_i(\tau+h)-g_i(\tau)\| \tau^{\theta''} d\tau\|Au\|_{\theta''}
\leq  Ch^{\theta'+\varepsilon}\|Au\|_{\theta'}, 0\leq t\leq h<1. \nonumber
\end{align}
In the case when $t>h$, from the above calculation, we have the following estimate
\begin{align}\label{I_9}
\|I_{32}^h(t)\|_{L^2(\Omega)}
\leq Ch^{\theta'+\varepsilon}\|Au\|_{\theta'} 
+\sum_{i=2}^{\ell}\int_h^t \|(g_i(\tau+h)-g_i(\tau))A(q_iu(t-\tau)-q_iu(t))\|_{L^2(\Omega)}d\tau.
\end{align}
After the change of variables, we represent $g_i(\tau+h)-g_i(\tau)$ in the form:
\begin{align}
{\Gamma(1-\alpha_i)}(g_i(\tau+h)-g_i(\tau))
=&\left((\tau+h)^{1-\alpha_i}-\tau^{1-\alpha_i}\right)
  \int_0^{1} A^{-1} S''\big((1-\eta)(\tau+h)\big)(\eta^{-\alpha_i}-1)d\eta\nonumber\\
 &-\tau^{1-\alpha_i} 
 \int_0^{1} A^{-1}\int_{\tau-\eta \tau}^{(1-\eta)(\tau+h)} S'''(\xi)d\xi(\eta^{-\alpha_i}-1)d\eta.\nonumber
\end{align}
Further, we shall need the inequalities:
\begin{align}
(\tau+h)^{1-\alpha_i}-\tau^{1-\alpha_i}\leq Ch\tau^{-\alpha_i},\forall \eta\in[h,t],\nonumber\\
\|A^{-1}S'''(\xi)\|\leq C\xi^{\alpha_1-3},\forall \xi>0.\nonumber
\end{align}
Then we have
\begin{align}
&{\Gamma(1-\alpha_i)}\|g_i(\tau+h)-g_i(\tau)\|\nonumber\\
\leq&Ch\tau^{-\alpha_i}
  \int_0^{1} \big((1-\eta)(\tau+h)\big)^{\alpha_1-2}(\eta^{-\alpha_i}-1)d\eta
 +C\tau^{1-\alpha_i} 
 \int_0^{1} \int_{\tau-\eta \tau}^{(1-\eta)(\tau+h)} \xi^{\alpha_1-3}d\xi(\eta^{-\alpha_i}-1)d\eta\nonumber\\
\leq&\ Ch\tau^{\alpha_1-\alpha_i-2}
  \int_0^{1} (1-\eta)^{\alpha_1-2}(\eta^{-\alpha_i}-1)d\eta
\leq Ch\tau^{\alpha_1-\alpha_i-2},\ h\leq \tau\leq t. \nonumber
\end{align}

Consequently
$
\|g_i(\tau +h)-g_i(\tau)\|
\leq Ch\tau^{\alpha_1-\alpha_i-2}, \forall \tau\geq h.
$
Applying this estimate to (\ref{I_9}), for $t>h$, $0<h<1$, we deduce that
\begin{align}
\|I_{32}^h(t)\|_{L^2(\Omega)}
\leq Ch^{\theta'+\varepsilon}\|Au\|_{\theta'} 
+ C\sum_{i=2}^{\ell}\int_h^{\infty} h\tau^{\alpha_1-\alpha_i+\theta''-2} d\tau\|Au\|_{\theta''}
\leq Ch^{\theta'+\varepsilon}\|Au\|_{\theta'}.\nonumber
\end{align}
Collecting the above estimates, we have 
\begin{align}
\|A\mathcal{K}u(t+h)-A\mathcal{K}u(t)\|_{L^2(\Omega)}
\leq Ch^{\theta}\|F\|_{\theta}
+Ch^{\theta'+\varepsilon}\|Au\|_{\theta'},\forall \theta'\in[0,\theta).
\end{align}
This complete the proof of the lemma.
\end{proof}

\begin{theorem}[H\"{o}lder estimate] \label{Holder} Suppose that $a=0$ and $F\in X_{\theta}$ with $\theta\in(0,1)$. 
Then the initial-boundary value problem (\ref{multifrac}) admits a unique mild solution $u\in X_{\theta}$.
Moreover there exists a constant $C>0$ such that
\[
\|Au\|_{\theta} \leq C \|F\|_{\theta}.
\]
\end{theorem}
\begin{proof}
Since $C^{\theta}([0,T];L^2(\Omega))\subset L^{\infty}(0,T;L^2(\Omega))$, 
and applying the Theorem \ref{regularity results}, we find a unique
$u\in C([0,T];D(A^{\gamma}))$ for each $\gamma\in(0,1)$
such that $u=\mathcal{K}u$, and for any $\gamma\in(0,1)$, there exists constant $C>0$ such that
\begin{align} \label{u<F}
|u\|_{C([0,T];D(A^{\gamma}))}\leq C\|F\|_{C([0,T];L^2(\Omega))}.
\end{align}
Moreover, we can obtain $u\in C([0,T];H^2(\Omega))$ and the following estimate
$
\|u\|_{C([0,T];H^2(\Omega))} \leq C\|F\|_{\theta},
$
that is $\|Au\|_{C([0,T];L^2(\Omega))}\leq C\|F\|_{\theta}$.

By Lemma \ref{Ku(t+h)-Ku(t)}, we can see that for $0\leq t<t+h\leq T$, $0<h<1$,
\begin{align}\label{Au(t+h)-Au(t)}
\|Au(t+h)-Au(t)\|_{L^2(\Omega)}
\leq Ch^{\theta}\|F\|_{\theta}
+Ch^{\theta'+\varepsilon}\|Au\|_{\theta'},\forall \theta'\in[0,\theta).
\end{align}

Let $\theta'=0$ in (\ref{Au(t+h)-Au(t)}). We have
$
\|Au\|_{\varepsilon}\leq C(\|F\|_{\theta}+\|Au\|_0)\leq C\|F\|_{\theta}.
$
Next, let $\theta'=\varepsilon$. Repeating the argument in the above, we have
$
\|Au\|_{2\varepsilon}\leq C(\|F\|_{\theta}+\|Au\|_{\varepsilon})\leq C\|F\|_{\theta}.
$
Step by step, it is clear from the above argument that
$
\|Au\|_{\theta}\leq C\|F\|_{\theta}.
$
Then theorem is thus proved.
\end{proof}


\section{Analyticity}
We set $F=0$ in initial-boundary value problem (\ref{multifrac}):
 \begin{equation} \label{multifrac F=0}
   \left\{ {\begin{array}{*{20}c}
  { \partial_t^{\alpha_1} u(t) + \sum_{j=2}^{\ell}q_j(x) \partial_t^{\alpha_j} u(t)
   =-\mathcal{A}u(t) + B(x)\cdot \nabla u(t)},\ 0<t\leq T, \hfill \\
  {u(x,0)=a,\ x\in \Omega,} \hfill \\
  {u(x,t)=0,\ x\in \partial\Omega,\ t \in (0,T).} \hfill
 \end{array} } \right.
 \end{equation}
According to the results in Theorem \ref{regularity results}, we see that the solution to (\ref{multifrac F=0}) satisfies the 
integral equation (\ref{solution aF}). After the change of variables, we have
\begin{align}\label{Standard form of solution}
u(t)
&=S(t)a-\sum_{i=2}^{\ell}\frac{t^{1-\alpha_i}}{\Gamma(1-\alpha_i)}
  \int_0^1 A^{-1}S'\big((1-\tau)t\big)\tau^{-\alpha_i}q_iad\tau 
-t\int_0^1 A^{-1}S'(\tau t)B\cdot\nabla u\big((1-\tau)t\big)d\tau \nonumber\\
&+\sum_{i=2}^{\ell}\frac{t^{1-\alpha_i}}{\Gamma(1-\alpha_i)}
    \int_0^1 A^{-1}S'(\tau t)\tau^{-\alpha_i}q_iu\big((1-\tau)t\big)d\tau \nonumber\\
&+\sum_{i=2}^{\ell}\frac{t^{2-\alpha_i}}{\Gamma(1-\alpha_i)}\int_0^1 \int_0^1 A^{-1}S''\big((1-\eta)\tau t\big)(\eta^{-\alpha_i}-1)
 q_iu\big((1-\tau)t\big)\tau^{1-\alpha_i}d\eta d\tau.
\end{align}
Moreover, we extend the variable $t$ in (\ref{solution aF}) from $(0,T)$ to the sector $\mathcal{S}:=\{z\neq0; |\arg z|<\frac{\pi}{2}\}$, and setting $u_0=0$, we define $u_{n+1}(z)(n=0,1,\cdots)$ , $z\in\mathcal{S}$ as follows:
\begin{align} \label{induction u_n(z)} 
u_{n+1}(z)
=&\sum_{i=2}^{\ell}
     \frac{-z^{1-\alpha_i}}{\Gamma(1-\alpha_i)}\int_0^1 A^{-1}S'\big((1-\tau)z\big)\tau^{-\alpha_i}q_iad\tau 
-z\int_0^1 A^{-1}S'(\tau z)B\cdot\nabla u_n\big((1-\tau)z\big)d\tau \nonumber\\
&+S(z)a
+\sum_{i=2}^{\ell}\frac{z^{1-\alpha_i}}{\Gamma(1-\alpha_i)}
   \int_0^1 A^{-1}S'(\tau z)\tau^{-\alpha_i}q_iu_n\big((1-\tau)z\big)d\tau \nonumber\\
&+\sum_{i=2}^{\ell}\frac{z^{2-\alpha_i}}{\Gamma(1-\alpha_i)}
    \int_0^1 \int_0^1 A^{-1}S''\big((1-\eta)\tau z\big)(\eta^{-\alpha_i}-1)q_iu_n\big((1-\tau)z\big)
      \tau^{1-\alpha_i}d\eta d\tau.
\end{align}
From the definition of (\ref{S(t)}) and the property of Mittag-Leffler function, we see that for any $n\in\mathbb{N}$, $u_n(z)$ is analytic in the sector $\mathcal{S}$. Moreover the following estimate holds.
\begin{lemma} \label{induction estimate of Au_n(z)}
For any constants $0<\theta<\frac{\pi}{2}$ and $T>0$, there exist constants 
$M>0$ and $M_1>0$ such that the following estimate 
\begin{align} \label{induction estimate complex Au_n}
\|A^{\frac{1}{2}}u_{n+1}(z)-A^{\frac{1}{2}}u_n(z)\|_{L^2(\Omega)}
\leq M_1\frac{M^n|z|^{\bar{\alpha}n-\frac{\alpha_1}{2}}}
          {\Gamma(\bar{\alpha}n+1-\frac{\alpha_1}{2})}\|a\|_{L^2(\Omega)},\ n\in\mathbb{N}
\end{align}
holds for $\forall z\in\mathcal{S}_{T}^{\theta}:= \{z\in\mathcal{S}; |\Re z|\leq T, |\arg z|\leq \theta\}$, 
where 
$\bar{\alpha}:=\min_{i=2,\cdots,\ell}\{\alpha_1-\alpha_i,\frac{\alpha_1}{2}\}$.
\end{lemma}
Here and henceforth, the constants $C$, $M_1$, $M$ denote the constants which are independent of $n$, $a$ and $u$, but may depend on $\{\alpha_i\}_{i=1}^{\ell}$, $d$, $\Omega$, $\theta$, $T$, $\{q_j\}_{j=2}^{\ell}$ and the coefficients of $\mathcal{A}$.
\begin{proof}
Firstly, for $n=0$, by (\ref{Ar-1s'}) and (\ref{Ar-1s''}), noting that $|z|^{\frac{\alpha_1}{2}-\alpha_i}\leq C|z|^{-\frac{\alpha_1}{2}}$, $z\in\mathcal{S}_{T}^{\theta}$, we have
\begin{align*}
\|A^{\frac{1}{2}}(u_{1}(z)-u_0(z))\|_{L^2(\Omega)}
&\leq C|z|^{-\frac{\alpha_1}{2}}\|a\|_{L^2(\Omega)}
+C\sum_{i=2}^{\ell}|z|^{\frac{\alpha_1}{2}-\alpha_i}
    \int_0^1 (1-\tau)^{\frac{\alpha_1}{2}-1}\tau^{-\alpha_i}d\tau\|a\|_{L^2(\Omega)}\\
&\leq C|z|^{-\frac{\alpha_1}{2}}\|a\|_{L^2(\Omega)}.
\end{align*}

Next, for any $n\in\mathbb{N}$, taking the operator $A^{\frac{1}{2}}$ on the both side of (\ref{induction u_n(z)}), by the induction assumption and the (\ref{Ar-1s'}) and (\ref{Ar-1s''}) for the $z\in \mathcal{S}$, similarly to the argument in the Theorem \ref{Estimate}, we can prove that
\begin{align*}
&\|u_{n+1}(z)-u_n(z)\|_{D(A^{\frac{1}{2}})}
\leq CM_1\frac{M^{n-1}|z|^{\bar{\alpha}n-\tfrac{\alpha_1}{2}}}{\Gamma(\bar{\alpha}(n-1)+1-\tfrac{\alpha_1}{2})}
\int_0^1 \tau^{\bar{\alpha}-1} {(1-\tau)^{\bar{\alpha}(n-1)-\tfrac{\alpha_1}{2}}}d\tau \|a\|_{L^2(\Omega)} \\
=&CM_1\frac{M^{n-1}|z|^{\bar{\alpha}n-\tfrac{\alpha_1}{2}}}{\Gamma(\bar{\alpha}(n-1)+1-\tfrac{\alpha_1}{2})}
B(\bar{\alpha},\bar{\alpha}(n-1)+1-\tfrac{\alpha_1}{2})\|a\|_{L^2(\Omega)} \\
=&CM_1\Gamma(\bar{\alpha})M^{n-1}\frac{|z|^{\bar{\alpha}n-\tfrac{\alpha_1}{2}}}{\Gamma(\bar{\alpha}n+1-\tfrac{\alpha_1}{2})}\|a\|_{L^2(\Omega)}
=M_1\frac{M^n|z|^{\bar{\alpha}n-\frac{\alpha_1}{2}}}
         {\Gamma(\bar{\alpha}n+1-\frac{\alpha_1}{2})}\|a\|_{L^2(\Omega)},
\end{align*}
where we set $M:=C\Gamma(\bar{\alpha})$, $M_1:=C\Gamma(1-\frac{\alpha_1}{2})$.
\end{proof}

\begin{theorem} \label{analytic F=0} 
Suppose that $u\in C([0,T];L^2(\Omega))\cap C((0,T];H^2(\Omega)\cap H_0^1(\Omega))$ is the mild solution to
(\ref{multifrac F=0}), then $u:(0,T]\rightarrow H_0^1(\Omega)$ can be analytically extended to a sector
$\mathcal{S}=\{ z\neq0;|\arg z|<\frac{\pi}{2} \}$. 
\end{theorem}
\begin{proof}
For any $\delta>0$, we denote $\mathcal{S}_{\delta,T}^{\theta}:=\{z\in\mathcal{S}_{T}^{\theta};|z|\geq\delta\}$.
From the definition of $u_n(z)$ in {(\ref{induction u_n(z)})}, it is easy to show that 
$
A^{\frac{1}{2}}u_n(z):\mathcal{S}_{\delta,T}^{\theta}\rightarrow L^2(\Omega)
$
is analytic. Hence according to Lemma \ref{induction estimate of Au_n(z)}, 
there exists $A^{\frac{1}{2}}\tilde{u}(z)\in L^2(\Omega)$ such that 
$\|A^{\frac{1}{2}}u_n(z) - A^{\frac{1}{2}}\tilde{u}(z)\|_{L^2(\Omega)}$ uniformly tends to $0$,
 $z\in\mathcal{S}_{\delta,T}^{\theta}$, as $n\rightarrow \infty$.
Therefore $A^{\frac{1}{2}}\tilde{u}(z)$ is analytic in $\mathcal{S}_{\delta,T}^{\theta}$. 
Moreover, since $\delta$, $T$, $\theta$ are arbitrarily chosen, then we deduce $A^{\frac{1}{2}}\tilde{u}(z)$
is actually analytic in the sector $\mathcal{S}$.

Finally, we show that $\tilde{u}(z)$ is just the mild solution $u$ to (\ref{multifrac F=0}) when the variable
$z$ is restricted on the interval $(0,T)$. In fact, denoting the imaginary part of $\tilde{u}(t)$, 
$\forall t\in (0,T)$ as $\Im u(t)$, we see that $\Im u(t)$ is the mild solution to the following 
initial-boundary problem:
\begin{equation}
   \left\{ {\begin{array}{*{20}c}
  { \partial_t^{\alpha_1} u(t) + \sum_{j=2}^{\ell}q_j(x) \partial_t^{\alpha_j} u(t)
   =-\mathcal{A}u(t) + B(x)\cdot \nabla u(t)},\ t>0, \hfill \\
  {u(x,0)=0,\ x\in \Omega,} \hfill \\
  {u(x,t)=0,\ x\in \partial\Omega,\ t \in (0,T).} \hfill
 \end{array} } \right.\nonumber
 \end{equation}
Using the uniqueness result of the above problem, we have $\Im u(t)=0$, $\forall t\in(0,T)$.
So that again by the uniqueness argument we see that $\tilde{u}(t)=u(t)$, $\forall t\in(0,T)$. This completes the proof of the theorem.
\end{proof}


\section{The weak unique continuation}
For the parabolic equation, there are a well known principle called unique continuation, generally speaking,
any solution of a parabolic equation that is defined on a domain $D$ must vanish in all of $D$ if it vanishes
 on an open set in $D$(See, e.g., \cite{Choulli}). Here for the fractional diffusion equation, we are to establish similar result to the parabolic equation.

\begin{theorem}[Weak unique continuation]\label{WUC}
Assuming that $0<\alpha_{\ell}<\cdots<\alpha_2<\alpha_1<1$, $a\in L^2(\Omega)$. Let $q_j$ be constant for 
$j=2,\cdots, \ell$. Suppose that $u\in C([0,T])\cap C((0,T];H^2(\Omega)\cap H_0^1(\Omega))$ is the mild solution to the following initial-boundary value problem: 
\begin{equation} \label{q_j constant}
   \left\{
  {\begin{array}{*{20}c}
  { \partial_t^{\alpha_1} u(x,t) + \sum_{j=2}^{\ell} q_j \partial_t^{\alpha_j} u(x,t)
   =-\mathcal{A}u(x,t) },\ 0<t\leq T, \hfill \\
  {u(x,0)=a(x),\ x\in \Omega,} \hfill \\
  {u(x,t)=0,\ x\in \partial\Omega,\ t \in (0,T].} \hfill
 \end{array} } \right.
 \end{equation}
Let $\omega\subset\Omega$ be an arbitrarily chosen subdomain and let $T>0$. Then $u(x,t)=0$, $x\in\omega$,
 $0<t<T$, implies $u=0$ in $\Omega\times(0,T)$.
\end{theorem}
\begin{proof}
According to Theorem \ref{regularity results}, we can uniquely extend the existence interval of $u$ into $[0,\infty)$. Therefore, we can describe the solution $u(t)$ to (\ref{q_j constant}) by the Laplace transform:  
\[
\mathcal{L}(u)(x,s) = \sum_{n=1}^{\infty} h_n(s)(a,\phi_n)\phi_n(x),
\quad x\in \Omega, \thinspace \Re{s} > M.
\]
Here $M>0$ is a constant such that the Laplace transform of $u$ converges for $\Re{s}>M$.
\[
h_n(s) = \frac{s^{-1}
\left(s^{\alpha_1} + \sum_{j=2}^{\ell}q_js^{\alpha_j}
\right)}
{s^{\alpha_1} + \sum_{j=2}^{\ell}q_js^{\alpha_j}
+ \lambda_n}.
\] 
Let $\omega \subset \Omega$ be an arbitrarily fixed sub domain. According to the Theorem \ref{analytic F=0},
we have $u:(0,\infty)\rightarrow L^2(\Omega)$ can be analytically extended to a sector 
$\mathcal{S}=\{ z\neq0;|\arg z|<\frac{\pi}{2} \}$. From the definition for the analytic of vector-valued
function, we see that
\[
\lim_{\Delta z\rightarrow0} \frac{u(z+\Delta z)-u(z)}{\Delta z}
\]
exists(in the topology of $L^2(\Omega)$), for any $z\in \mathcal{S}$. Then $u:\mathcal{S}\rightarrow L^2(\omega)$
is also analytic.
Therefore, $u:\mathcal{S}\rightarrow L^2(\omega)$ is also weakly analytic
(see, e.g., Theorem 3.31 on p.82 in \cite{Rudin1}). 
That is, for any $\varphi\in L^2(\omega)$, the function $(u(z),\varphi)
_{L^2(\omega)}$ is analytic in the ordinary sense. $u|_{\omega\times(0,T)}=0$ 
derives $(u(t),\varphi)_{L^2(\omega)}=0$ for
any $t\in(0,T)$. By analyticity, we recognize that $(u(t),\varphi)_{L^2(\omega)}=0$ for any $t\in(0,\infty)$. 
Since $\varphi$ is chosen arbitrarily, so $u=0$ in $\omega\times(0,\infty)$. 
Therefore
\[
\mathcal{L}(u)(s)
=\sum_{n=1}^{\infty}
\frac{s^{\alpha_1} + \sum_{j=2}^{\ell}q_js^{\alpha_j}}
{s^{\alpha_1} + \sum_{j=2}^{\ell}q_js^{\alpha_j} + \lambda_n}
(a,\phi_n)\phi_n = 0\ \text{in}\ \omega, \thinspace \Re{s} > M.
\]
Setting $\eta = s^{\alpha_1} + \sum_{j=2}^{\ell}q_js^{\alpha_j}$, we 
see that $\eta$ varies over some domain $E \in \mathbb{C}$ as 
$s$ varies over $\Re{s} > M$.
Therefore
\begin{align} \label{identity zero}
 \sum_{n=1}^{\infty}
\frac{1}{\eta + \lambda_n}(a,\phi_n)\phi_n(x) = 0,\thinspace \text{in}\ \omega,\ \eta \in E.
\end{align}
We set $\sigma(A)=\{\mu_k\}_{k\in\mathbb{N}}$ and we denote by $\{\varphi_{kj}\}_{1\leq j\leq m_k}$ an orthonormal 
basis of Ker$(\mu-A)$. Note that we regards $\sigma(A)$ as set, not as sequence with multiplicities. Therefore we can rewrite (\ref{identity zero}) by
\begin{align} \label{identity zero multi}
 J(\eta):=\sum_{k=1}^{\infty} \sum_{j=1}^{m_k}
\frac{1}{\eta + \mu_k}(a,\phi_n)\varphi_{kj} = 0\ \text{in}\ \omega,\thinspace \forall \eta \in E.
\end{align}
It is easy to see that $J:E\rightarrow L^2(\Omega)$ is analytic. Moreover, we can derive that $J(\eta)$
defined in (\ref{identity zero multi}) holds for $\eta\in\mathbb{C}\backslash\{-\mu_k\}_{k\in\mathbb{N}}$. 
We can take a suitable disk which includes $-\mu_k$ and does not include $\{-\mu_i\}_{i\neq k}$. 
Integrating (\ref{identity zero multi}) in this disk, we have
 \[
u_k=\sum_{j=1}^{m_k} (a,\varphi_{kj})\varphi_{kj}=0\ \text{in}\ \omega.
 \]
Since $(A-\mu_k)u_k=0$ and $u_k=0$ in $\omega$, the unique continuation (e.g., Isakov \cite{Isakov}) implies 
$u_k=0$ in $\Omega$ for each $k\in\mathbb{N}$. Since $\{\varphi_{kj}\}_{1\leq j\leq m_k}$ is linearly independent in $\Omega$, we see that $(a,\varphi_{kj})=0$ for $1\leq j\leq m_k$, $k\in\mathbb{N}$. 
Therefore $u=0$ in $\Omega\times(0,T)$. This completes the proof of the theorem.
\end{proof}

\begin{corollary}
Let $\Gamma$ be an open subset of $\partial\Omega$, and $u\in C([0,T];L^2(\Omega))\cap C((0,T];H^2(\Omega)\cap H_0^1(\Omega))$ satisfy the problem (\ref{q_j constant})
and $\partial_{\nu}u|_{\Gamma\times(0,T)}=0$. Then $u=0$ in $\Omega\times(0,T)$.
\end{corollary}
\begin{proof}
Since the boundary $\partial\Omega$ of $\Omega$ is smooth enough, we can choose an open set $\omega$ such that
$\omega\cap\Omega=\Gamma$ and the boundary of the new domain $\widetilde{\Omega}:=\Omega\cup\omega$ is smooth enough. We set
\begin{align}
\widetilde{u}(x,t)
:=\left\{
\begin{array}{ll}
    u(x,t), &\forall (t,x)\in\Omega\times(0,T), \\
    0,      &\forall (t,x)\in\omega\times(0,T). \\
\end{array}
\right.
\end{align}
According to the condition $u|_{\partial\Omega\times(0,T)}=\partial_{\nu}u|_{\Gamma\times(0,T)}=0$, 
it is easy to see that the new function $\widetilde{u}$ belongs to 
$C([0,T];L^2(\widetilde{\Omega}))\cap C((0,T];H^2(\widetilde{\Omega})\cap H_0^1(\widetilde{\Omega}))$.

On the other hand, the proof of the weak unique continuation is divided into the following steps:
Step 1. Extending existence interval of the solution to $(0,\infty)$ and taking the fractional integral operator $J^{\alpha_1}$ on the both side of the equation, that is,
\begin{equation} \label{J^alpha_1 q_j constant}
   \left\{
  {\begin{array}{*{20}c}
  { u-a + \sum_{j=2}^{\ell} q_j J^{\alpha_1-\alpha_j} (u-a)
   =-\mathcal{A}J^{\alpha_1}u },\ x\in\Omega,\ t>0, \hfill \\
  {u=0,\ x\in \partial\Omega,\ t>0.} \hfill
 \end{array} } \right.
 \end{equation}
Step 2. Taking Laplace transform on the both side of the above equation, we can obtain the 
weak unique continuation by some results in the complex analysis.

Using the fact that $u$ is also the solution of problem (\ref{J^alpha_1 q_j constant}), then $\widetilde{u}$ 
is the solution of the following problem
\begin{equation}
   \left\{
  {\begin{array}{*{20}c}
  { \widetilde{u}-\widetilde{a} + \sum_{j=2}^{\ell} q_j J^{\alpha_1-\alpha_j} (\widetilde{u}-\widetilde{a})
   =-\mathcal{A}J^{\alpha_1}\widetilde{u} },\ t>0, \hfill \\
  {\widetilde{u}=0,\ x\in \partial\widetilde{\Omega},\ t>0,} \hfill
 \end{array} } \right.
 \end{equation}
where
\begin{align}
\widetilde{a}(x)
:=\left\{
\begin{array}{ll}
    a(x), &\forall x\in\Omega, \\
    0,    &\forall x\in\omega. \\
\end{array}
\right.
\end{align}
From Theorem \ref{WUC}, we deduce $\widetilde{u}=0$ in $\widetilde{\Omega}\times(0,T)$, that is, 
$u=0$ in $\Omega\times(0,T)$. This completes the proof of corollary.
\end{proof}

\section{Conclusions}
Initial-boundary value problem for the linear diffusion equation with multiple time-fractional
derivatives was investigated. We proved unique existence of the solution by the eigenfunction expansion 
and Laplace transform as well as the H\"{o}lder regularity and related properties. 
\\

\noindent (i) In the case where the equation with single fractional time-derivative, the solution
to the initial-boundary value problem in \cite{SakYam} can easily achieve the $C((0,T];H^2(\Omega))$ regularity, while we only proved that the solution is in $C((0,T];H^{2\gamma}(\Omega))$, $\gamma\in(0,1)$ (see Theorem 
\ref{regularity results}), and we do not know whether the solution can achieve 
the regularity as in \cite{SakYam}. 
\\

\noindent (ii) We only proved the H\"{o}lder's regularity with the initial condition being zero. We guess that there may be a H\"{o}lder regularity similar to Theorem 2.4 in \cite{SakYam}. 
\\

\noindent (iii) For the classical diffusion equation, we have that the unique continuation principle holds without $u|_{\partial\Omega\times(0,T)}=0$ (e.g., 
\cite{Isakov}). However for our case, we do not know whether the uniqueness holds without such kind of assumption.



\begin{thebibliography}{79}
   \bibitem{Adams}
     R. A. Adams, {\em Sobolev Spaces}, Academic Press, New York, {1999}.
     
   \bibitem{AdGe}
   E. E. Adams, L. W. Gelhar, Field study of dispersion in a heterogeneous aquifer 2: Spatial moments
   analysis, Water Resources Research 28(1992) 3293-3307.
     
   \bibitem{SB}
   S. Beckers and M. Yamamoto, Regularity and unique existence of solution to linear diffusion equation withmultiple time-fractional derivatives. {\em Control and Optimization with PDE Constraints (2013)} ed K. Bredies, C. Clason, K. Kunisch and G. von Winckel (Basel: Birkh\"{a}user).
   
   \bibitem{BSS}
   B. Berkowitz, H. Scher and S. E. Silliman, Anomalous transport in laboratory-scale heterogeneous porous
   media. Water Resource Research 36(2000) 149-158.
   
   \bibitem{BSMW}
   D. A. Benson, R. Schumer, M. M. Meerschaert and S. W. Wheatcraft, Fractional dispersion, levy motion, and 
   the MADE tracer tests. Transport in Porous Media 42(2001) 211-240.   
   
   \bibitem{Chechkin1}
   A.V. Chechkin, R. Gorenflo, I.M. Sokolov, Retarding subdiffusion and accelerating superdiffusion 
   governed by distributed order fractional diffusion equations, Phys. Rev. E 66(2002) 1-7.
   
   \bibitem{Chechkin2}
   A.V. Chechkin, R. Gorenflo, I.M. Sokolov, V.Yu. Gonchar, Distributed order time fractional diffusion
   equation, Fract. Calc. Appl. Anal. 6(2003) 259-279.
     
   \bibitem{Choulli}
     M. Choulli, {\em Une Introduction aux Problems Inverses 
Elliptiques et Paraboliques}.
     Springer-Verlag, { 2009}.
    
   \bibitem{DiLu}
   K. Diethelm and Y. Luchko, Numerical solution of linear multi-term initial value problems of fractional
   order. J. Comput. Anal. Appl. 6 (2004), 243-263.
   
   \bibitem{DGB}  
   V. Daftardar-Gejji, S. Bhalekar, Boundary value problems for multi-term fractional differential equations,
   J. Math. Anal. Appl. 345 (2008) 754-765.
   
   \bibitem{GGR}
   M. Giona, S. Gerbelli and H. E. Roman, Fractional diffusion equation and relaxation in complex
   viscoelastic materials. Physica A 191(1992) 449-453.
   
   \bibitem{GLZ}
   R. Gorenflo, Y. Luchko and P. P. Zabrejko, On solvability of linear fractional differential equations in 
   Banach spaces. Frac. Calc. Appl. Anal. 2(1999) 163-176.
   
   \bibitem{Hanyaga}
   A. Hanyaga, Multidimensional solutions of time-fractional diffusion-wave equations. Proc. R. Soc. Lond.
   (Ser. A) 458(2002) 933-957.
   
   \bibitem{Hatano}
   Y. Hatno, N. Hatano, Dispersive transport of ions in column experiments: an explanation of long-tailed
   profiles, Water Resources Res. 34(1980) 1027-1033.

   \bibitem{Henry}
     D. Henry, {\em Geometric Theory of Semilinear Parabolic Equations}. Springer-Verlag, Berlin, 1981.
     
   \bibitem{Isakov}
   V. Isakov, {\em Inverse Problems for Partial Differential Equations}. Springer-Verlag, New York, 1998.
     
   \bibitem{JLTB}
H. Jiang, F. Liu, I. Turner and K. Burrage, Analytical solutions for the 
multi-term time-space 
   Caputo-Riesz fractional advection-diffusion equations on a finite 
domain. J. Math. Anal. Appl. 389 (2012),   1117-1127.
     
   \bibitem{Kil}
   A. A. Kilbas, H. M. Srivastava and J. J. Trujillo, Theory and Applications of Fractional Differential
   Equations, Amsterdam: Elsevier, 2006.
     
   \bibitem{Kochubei}
   A.N. Kochubei, Distributed order calculus and equations of ultraslow diffusion, J. Math. Anal. Appl.
   340(2008) 252-281.
     
   \bibitem{Lions}
J. L. Lions; E. Magenes, {\em Non-homogeneous Boundary Value Problems and Applications}.
     Springer-Verlag, {1972}.
   
   \bibitem{Luchko00}
   Y. Luchko, R. Gorenflo, An operational method for solving fractional differential equations with the
   Caputo derivarives, Acta Math. Vietnam. 24(1999) 207-233.
   
   \bibitem{Luchko0}
   Y. Luchko, Boundary value problems for the generalized time-fractional diffusion equation of distributed
   order, Fract. Calc. Appl. Anal. 12(2009) 409-422.
     
   \bibitem{Luchko}
   Y. Luchko, Maximum principle for the generalized time-fractional diffusion equation. 
   J. Math. Anal. Appl. 351 (2009), 218-223.
   
   \bibitem{Luchko1}
   Y. Luchko, Some uniqueness and existence results for the initial-boundary-value problems for the
   generalized time-fractional diffusion equation. Computers and Mathematics with Applications 59 (2010),
   1766-1772.
   
   \bibitem{Luchko2}
   Y. Luchko, Initial-boundary-value problems for the generalized multi-term time-fractional diffusion
   equation. J. Math. Anal. Appl. 374 (2011), 538-548.
   
   \bibitem{Meer}
   M.M. Meerschaert, E. Nane, P. Vellaisamy, Distributed-order fractional Cauchy problems on bounded domains,
   arXiv:0912.2521v1.
   
   \bibitem{Meerschaert2}
   M. M. Meerschaert, E. Nane, H.P. Scheffler, Stochastic model for ultraslow diffusion, Stochastic Process. 
   Appl 116(2006)1215-1235.
   
   \bibitem{MK}
   R. Metzler and J. Klafter, Boundary value problems for fractional diffusion equations. Physica A 278(2000)
   107-125.
   
   \bibitem{MT}
   M. M. Meerschaert and C. Tadjeran, Finite difference approximations for fractional advection-dispersion flow equations. Journals of Computational and Applied Mathematics 172(2004) 65-77.
     
   \bibitem{Pazy}
     A. Pazy, {\em Semigroups of Linear Operators and Applications to Partial Differential Equations}. 
     Springer-Verlag, Berlin, { 1992}.
     
   \bibitem{Podlubny}
     I. Podlubny, {\em Fractional Differential Equations}. Academic Press, San Diego, {1999}.
     
   \bibitem{RA}
   H. E. Roman, P. A. Alemany, Continuous-time random walks and the fractional diffusion equation. 
   J, Phys. A 27(1994) 3407-3410.
   
   \bibitem{Rudin1}
   Water Rudin, {\em Functional Analysis 2nd ed.}, McGraw-Hill, { 1991}.

   \bibitem{SakYam}
     K. Sakamoto; M. Yamamoto, Initial value/boundary value problems for fractional diffusion-wave equations
     and applications to some inverse problems. J. Math. Anal. Appl. {\em 382}(2011), 426-447.
     
   \bibitem{Sokolov}
   L. M. Sokolov, A.V. Chechkin, J. Klafter, Distributed-order fractional kinetics, 
   Acta Phys. Polon. B 35(2004) 1323-1341.
   
   \bibitem{Umarov}
   S. Umarov, R. Gorenflo, Cauchy and nonlocal multi-point problems for distributed order pseudo-differential
   equations, Z. Anal. Anwend. 24(2005) 449-466.
   
   \bibitem{XCY}
   X. Xu, J. Cheng and M. Yamamoto, Carleman estimate for a fractional diffusion equation with half order and 
   application. Applicable Analysis 90(2011) 1355-1371.

 \end{thebibliography}
 \end{document}